\documentclass[sn-mathphys,a3paper]{sn-jnl}

\jyear{2021}%

\theoremstyle{thmstyleone}%
\newtheorem{theorem}{Theorem}
\newtheorem{proposition}[theorem]{Proposition}%

\theoremstyle{thmstyletwo}%
\newtheorem{remark}{Remark}%
\newtheorem{lemma}{Lemma}

\theoremstyle{thmstylethree}%

\usepackage{xcolor}
\usepackage{lipsum}
\usepackage{amsfonts}
\usepackage{graphicx}
\usepackage{amsmath}
\usepackage{amssymb}
\usepackage{mathtools}
\newcommand{\assign}{:=}
\newcommand{\nobracket}{}

\newcommand{\tmop}[1]{\ensuremath{\operatorname{#1}}}
\newcommand\numberthis{\addtocounter{equation}{1}\tag{\theequation}}
\def\off2{\textrm{off}_2}
\def\cD{\mathcal{D}}
\def\CC{\mathbb{C}}
\DeclarePairedDelimiter\abs{\lvert}{\rvert}
\usepackage{graphicx}
\usepackage{cleveref}
\usepackage{anyfontsize}

\def\dis{\displaystyle}

\begin{document}

\title[Newton-Type Methods For Simultaneous Matrix Diagonalization]{Newton-Type Methods For Simultaneous Matrix Diagonalization}


\author*[1]{\fnm{Rima} \sur{Khouja}}\email{rima.khouja@inria.fr}
\equalcont{These authors contributed equally to this work.}
\author[1]{\fnm{Bernard} \sur{Mourrain}}\email{bernard.mourrain@inria.fr}
\equalcont{These authors contributed equally to this work.}

\author[2]{\fnm{Jean-Claude} \sur{Yakoubsohn}}\email{yak@mip.ups-tlse.fr}
\equalcont{These authors contributed equally to this work.}

\affil*[1]{\orgdiv{AROMATH,
 INRIA Sophia Antipolis M{\'e}dit{\'e}rann{\'e}e},  \orgaddress{\street{2004, route des Lucioles}, \city{Sophia Antipolis}, \postcode{06902}, \country{France}}}

\affil[2]{\orgdiv{Institut de Math{\'e}matiques de Toulouse}, \orgname{Universit{\'e} Paul Sabatier}, \orgaddress{\street{118, route de Narbonne}, \city{Toulouse}, \postcode{31062},  \country{France}}}


\abstract{This paper proposes a Newton-type method to solve numerically the eigenproblem of several diagonalizable matrices, which pairwise commute. A classical result states that these matrices are simultaneously diagonalizable. From a suitable system of equations associated to this problem, we construct a sequence that converges quadratically towards the solution. This construction is not based on the resolution of a linear system as is the case in the classical Newton method. Moreover, we provide a theoretical analysis of this construction and exhibit a condition to get a quadratic convergence.  We also propose numerical experiments, which illustrate the theoretical results.}





\maketitle

\section{Introduction}

\subsection{Our study}

Let us consider $p$ {\em diagonalizable} matrices $M_1, \cdots, M_p$ in
$\CC^{n \times n}$ which pairwise commute. A classical result  states
that these matrices are simultaneously diagonalizable, i.e., there exists
an invertible matrix $E$ and diagonal matrices $\Sigma_i$, $1 \leqslant
i \leqslant p$, such that $EM_i E^{- 1} = \Sigma_i$, $1 \leqslant
i \leqslant p$, see e.g. \cite{HJ12}. The aim of this paper  is to compute numerically a solution $(E,F,\Sigma)$ of the system of equations
\begin{eqnarray}
  f(E,F,\Sigma):=\left(\begin{array}{c}
    FE - I_n\\
    FME - \Sigma
  \end{array}\right) & = & 0  \label{eq1}
\end{eqnarray}
where $\Sigma=(\Sigma_1,\ldots,\Sigma_p)$ and $F M E-\Sigma:=(FM_1 E -
\Sigma_1,\ldots,FM_p E - \Sigma_p)=0$. Notice that this system is
multi-linear in the unknowns $E,F, \Sigma$. We verify that when $p=1$
and $M_{1}$ is a generic matrix, this system has a solution set of dimension
$2\,n^{2} +n - 2n^{2} = n$ ($n^2+n^2+n$ unknowns for $E,F,\Sigma$ and $2$ matrix equations corresponding to $n^2+n^2$ equations). However, for $p>1$ and generic
matrices $M_{i}$, there is no solution. To have a solution, the pencil
$M$ must be on the manifold $\cD_{p}$ of $p$-tuples of simultaneously
diagonalizable matrices.

The system \eqref{eq1} can be generalized to the following system:
\begin{eqnarray}
  f'(E,F,\Sigma'):=\left(\begin{array}{c}
    FM_{0}E - \Sigma_{0}\\
    FME - \Sigma
  \end{array}\right) & = & 0  \label{eqh2}
\end{eqnarray}
where  $\Sigma'=(\Sigma_{0},\Sigma_1,\ldots,\Sigma_p)$, $M_{0}\in \CC^{n\times n}$ is replacing $I_{n}$ and $\Sigma_{0}$ is
a diagonal matrix replacing $I_{n}$ in the first equation of \eqref{eq1}.
When the pencil $M'=(M_{0}, M_{1}, \ldots, M_{p})$ contains an invertible
matrix, the solutions of the two systems are closely related. If
$M_{0}$ is invertible, a
solution $(E,F,\Sigma')$ of \eqref{eqh2} for $M'=(M_{0}, M_{1}, \ldots,
M_{p})$ gives the solution $(F M_{0}, E \Sigma_{0}^{-1}, \Sigma
\Sigma_{0}^{-1})$ of \eqref{eq1} for $M = (M_{0}^{-1} M_{1}, \ldots,
M_{0}^{-1} M_{p})$.
A similar correspondence between the solution sets can be obtained if
a linear combination $M_{0}'= \sum_{i=1}^{p} \lambda_{i} M_{i}$ is
invertible.

As \eqref{eqh2} can be seen as an homogeneisation
of \eqref{eq1} and appears in several contexts and applications, we
will also study Newton-type methods for this homogenized system.

To solve the system of equations \eqref{eq1}, we propose to apply a
Newton-like method and to analyze the Newton map associated to an iteration. These ideas have also been developed for instance in \cite{MAHONY199667} where a Newton
 method is used for the symmetric eigenvalue problem. A Simultaneous Newton's iteration
for ill-conditioned eigenproblem has been introduced in \cite{chatelin}. For more recent references using the Newton-type approach for eigenproblem see for instance \cite{joho2002joint,sato2017riemannian,joho2008newton}. Moreover, similar approach for the fast computation of the singular value decomposition has been presented in a technical report \cite{JVDHYak}.

We say that we have a quadratic sequence associated to a system of equations if the sequence converges quadratically towards a solution.

The classical Newton map  defines $(E+X,F+Y,\Sigma+S)$ from $(E,F,\Sigma)$
in order to cancel the linear part in the Taylor expansion of
$f(E+X,F+Y,\Sigma+S)$. An easy computation shows that the
perturbations $X$, $Y$ and $S$ are solutions of such a Sylvester-type linear system
\begin{eqnarray}
  \left(\begin{array}{c}
    FE - I_n+FX+YE\\
    F M E - \Sigma-S+XMF+EMY
  \end{array}\right) & = & 0.  \label{eq1_bis}
\end{eqnarray}
A straight-forward way to solve this linear system is via Kronecker product, see \cite{HJ91}. This leads to a linear system
of size $2n^2$, which can be solved in $\mathcal{O}(n^6)$ arithmetic operations.

The construction of the methods studied here is based on  perturbations of such type $(E(I_n+X), (I_n+Y) F, \Sigma +S)$ rather than  $(E+X,F+Y,\Sigma+S)$. More precisely the perturbations $X$, $Y$ and $S$ that we consider are perturbations which
cancel the linear part of the Taylor expansion of
$f(E(I_n+X),(I_n+Y)F,\Sigma+S)$. In this case, we can produce explicit solutions
for the linear system in $X$, $Y$ and $S$ given by:
\begin{eqnarray}
  \left(\begin{array}{c}
    Z+X+Y\\
    \Delta-S+\Sigma X+Y\Sigma
  \end{array}\right) & = & 0.  \label{eq1_ter}
\end{eqnarray}
where $Z = F E - I_{n}$ and $\Delta = F M E - \Sigma$.
We will see that the linear system \eqref{eq1_ter} admits an explicit solution
$(X,Y,S)$ with respect to $Z$ and $\Delta$ for $p=1,2$ in \eqref{eq1}. This is because $\Sigma$ is a diagonal matrix. From these considerations, we define and analyze a sequence that converges quadratically
towards a solution of the system \eqref{eq1} without inverting a linear
system at each step of this Newton-like method.
%
\subsection{Related works} Simultaneous matrix diagonalization  is
required by many algorithms as it was pointed out in \cite{laub1987computation,BBM92,vollgraf2006quadratic,jiang2016simultaneous,flury1994simultaneous}. A
numerical analysis  for two normal commuting matrices is proposed in
\cite{BBM93} using Jacobi-like methods. Their method adjusts the
classical Jacobi method in successively solving $\frac{n(n-1)}{2}$
two-real-variables optimization problems at each sweep of the
algorithm. Their main result states a local quadratic convergence and
can be summarized in the following way. Let $\off2 (A,B)^2=\sum_{i\ne j}\abs{A_{i,j}}^2+\abs{B_{i,j}}^2$. Let $\{\alpha_1, \dots, \alpha_n\}$ (resp. $\{\beta_1, \dots, \beta_n\}$) be the set of the eigenvalues of $A$ (resp. $B$). Let $A^k$ and $B^k$ the matrices obtained at the step $k$ of the Jacobi-like method and $\rho_k=\off2 (A^k,B^k)$. If
$$\rho_0<\frac{1}{2}\delta:=\frac{1}{4}\min_{i\ne j}{(\abs{\alpha_i-\alpha_j}, \abs{\beta_i-\beta_j})},$$
then
$$\rho_{k+1}<2n(9n-13)\,\frac{\rho^2_k}{\delta}.$$
We will see in Theorems \ref{theo2} and~\ref{th-quad-conv} that
the local conditions of the quadratic convergence do not depend on
$n$. Many other papers studied the so-called Jacobi-like methods (see
e.g. ~\cite{luc-alb},~\cite{mes-bel} and references therein).

In \cite{Hoeven} a sequence with proof of its convergence towards a
numerical solution of the system \eqref{eq1} when $p=1$, i.e., for
$M_1$, with the assumption of $M_1$ being a diagonalizable matrix, is
presented. It requires matrix inversion. Furthermore, under some extra
assumptions, its quadratic convergence is established.

For a pencil of real {\em symmetric} matrices $C=(C_1, \ldots, C_s)$, several
algorithms based on Riemannian optimization methods (see
\cite{AbsMahSep2008}) have been developed in order to find an {\em
  approximate joint diagonalizer} (see
e.g. \cite{bouchard,absil1,alg1,JM}). The idea is to find a local
minimizer $B\in\mathbb{R}^{n\times n}$ of an objective function $f$
which measures the degree of non-diagonality of the pencil $(BC_1B^T,
\ldots, BC_sB^T)$ over a Riemannian manifold (see
\cite{objfunc,bouchard,Afsari} for some examples of objective
functions). This Riemannian manifold is defined according to the
geometric constraints considered on $B$. For instance, the
diagonalizer is supposed to be orthogonal in some of these algorithms
after a pre-whitening step (see
e.g. \cite{blind,blind1,blind3,alg1,alg2,JM,alg3,alg4}). Due to
inaccuracies in the computation of the diagonalizer with orthogonality
constraints (see. \cite{inacc}), {\em oblique} constraints, i.e., all the rows of
the diagonalizer have unit Euclidean norm, have also been considered instead of the former constraints in more recent works
(see e.g. \cite{absil1,bouchard}). These algorithms can
be used when the pencil of symmetric matrices is simultaneously
diagonalizable. In this case, we aim to find a zero of the objective
function $f$. However, these algorithms have a computation
complexity higher than the Newton-type algorithm that we propose (see
Proposition \ref{complexity}). For instance, most of them  combine line search
\cite[Ch4]{AbsMahSep2008} or trust region \cite[Ch7]{AbsMahSep2008}
methods, and matrix inversions at each iteration (see the exact
Riemannian Newton iteration in \cite{absil1}). Moreover, the points on
the Riemannian manifold are updated using a retraction operator (see
\cite[Ch4]{AbsMahSep2008} or  \cite{bouchard} for an example of a
retraction operator on the oblique manifold). In the Newton-type
method described in \Cref{sec-p=1,sec-p=2} the points are updated by
using direct and explicit formulas. They have lower complexity than the Riemannian optimization-based algorithms and they are well-adapted to computation with high precision.

Simultaneous matrix diagonalization appears in many
applications.
For instance, in the solution of multivariate polynomial equations by algebraic
methods, the isolated roots of the system are obtained from the
computation of common eigenvectors of commuting operators of
multiplication in the quotient ring and from their eigenvalues
\cite{CoxUsingalgebraicgeometry2005},
\cite{elkadi_introduction_2007}. In the case of simple roots, this
reduces to simultaneous diagonalization of a pencil of matrices.

The approach of approximate joint diagonalizer for a pencil of real
{\em symmetric} matrices is used to solve Blind Source Separation
(BSS) problem, with potential applications in wide domains of
engineering (see e.g. \cite{BSS}).

Simultaneous matrix diagonalization of pencils of general matrices
also appears in the rank (or canonical) decomposition of tensors
\cite{lath06}. Under certain conditions this rank decomposition is
unique \cite{si-bro}. In this case simultaneous matrix diagonalization
allows to compute this rank decomposition which plays a crucial role
in numerous applications such that Psychometric \cite{ca-ch}, Signal
Processing and Machine Learning \cite{ci-ma}, \cite{si-lath}, Sensor
array processing \cite{so-do}, Arithmetic Complexity \cite{bu13}, wireless communications \cite{sovan}, multidimensional
harmonic retrieval \cite{so-lath17-1}, \cite{so-lath17-2},
Chemometrics \cite{bro97}, and Principal components
analysis \cite{jolli}.
\subsection{Outline}

Our \emph{contributions} are a new iteration for the simultaneous diagonalization of matrices, with a local quadratic convergence and its analysis. The iteration is different from a Newton iteration. It does not require to invert a large linear system, but performs simple matrix operations.
We analyse the numerical behavior of the method and provide a certification test for the convergence.
Sections \ref{sec-inverse-pb}, \ref{sec-p=1}, \ref{sec-p=2}, and \ref{sec-cvg-family} are devoted to respectively constructing a sequence to solve numerically:
\begin{itemize}
\item $FE- I_n=0$,
\item the system \eqref{eq1} when $p=1$,
\item the system \eqref{eqh2} when $p=1$,
\item the system \eqref{eq1} for any $p$.
\end{itemize}
Moreover, we provide for these cases, a certification that the sequence converges to a nearby solution, and a test to detect when this
convergence is quadratic from an initial point. 
More precisely, in Section~\ref{sec-p=1} we show that a triplet $(E_0,F_0,\Sigma_0)$ must satisfy a property depending on the quantity $\varepsilon_0 := \max (\kappa_0^2 K_0^2\| Z_0 \|,
  \kappa_0^2K_0\| \Delta_0 \|)$ to get a quadratic convergence where
\begin{itemize}

\item[1--] $Z_0=F_0E_0-I_n$,
\item[2--] $\Delta_0=F_0ME_0- \Sigma_0$,
\item[3--]  $\displaystyle \kappa_0=\max \left( 1, \;
       \max_{1 \leqslant j < k \leqslant n}  \frac{1}{\abs{ \nobracket \sigma_{0,
       k^{}}^{} - \sigma_{0, j}^{} }\nobracket} \right)$,
\item[4--] $       K_0 = \max_k \left( 1, \; \abs{ \nobracket
       \sigma_{0, k}^{} }\nobracket \right),$
\end{itemize}
where $\sigma_{0, 1}^{}, \ldots, \sigma_{0, n}^{}$ denote the diagonal
  entries of $\Sigma_0^{}$. The quantity $\kappa$ is the condition number of the studied methods. Based on the same methodology as in Section~\ref{sec-p=1}, Sections~\ref{sec-p=2} and \ref{sec-cvg-family} exhibit a certification of the convergence of the sequence constructed to the studied case towards the solution with a sufficient condition on the initial point.
 In Section~\ref{sec-exp} we perform numerical experimentation. The final section is for our conclusions and future works.
\subsection{Notation and preliminaries}
Throughout this work, we will use the infinity vector norm and the
corresponding matrix norm. For a given vector $v \in \mathbb{C}^n$ and matrix
$M \in \mathbb{C}^{n \times n}$, they are respectively given by:
\begin{eqnarray*}
  \| v \| & = & \max \{ \abs{v_1}, \ldots, \abs{v_n} \}\\
  \| M \| & = & \max_{\| v \| = 1}  \| M v \| .
\end{eqnarray*}
Explicitly, $\| M \|  = \max \{\abs{m_{i,1}}+\ldots+\abs{m_{i,n}}: 1\le i\le n\}.$\\
For a second matrix $N \in \mathbb{C}^{n \times n}, \text{we have}$
\begin{eqnarray*}
  \| M + N \| & \leqslant & \| M \| + \| N \| ~\text{(sub-additivity)}\\
  \| M N \| & \leqslant & \| M \|  \| N \| ~\text{(sub-multiplicativity)}.
\end{eqnarray*}
Moreover, for a given matrix $M\in\mathbb{C}^{n\times n}$, we denote by $\|M\|_{\mathrm{L}, \mathrm{Tri}}$\\and $\|M\|_{Frob}$ the following:
$$\|M\|_{\mathrm{L}, \mathrm{Tri}}:=\max_{\substack{{1\le i \le n}\\{1\le j\le i-1}}}\abs{m_{i,j}},$$ i.e the max matrix norm of the lower triangular part of $M,$\\
$$\|M\|_{Frob}:=\sqrt{\sum_{i=1}^{n}\sum_{j=1}^{n}{\abs{m_{i,j}}^2}},$$
i.e., the Frobenius norm of $M$.

Furthermore, we consider in this paper the regular case of diagonalizable matrices, that is, the matrices are diagonalizable with simple eigenvalues. Thus we will use the following notation
 $$\mathcal{W}_n:=\{M\in\mathbb{C}^{n\times n} \mid M\text{with pairwise distinct eigenvalues}\}.$$ It is well-known that $\mathcal{W}_n$ is dense in $\mathbb{C}^{n\times n}$.
 
The Lie group of $n\times n$ invertible matrices, denoted by $GL_n$, is the so-called general linear group \cite{lineargrp}. We denote by $\mathcal{D}_n$ the vector space of diagonal matrices of size $n$ and $\mathcal{D}_n'$ denotes the subset of $\mathcal{D}_n$ in which the diagonal matrices are of $n$ distinct diagonal entries. Let $E, F\in GL_n$ and $\Sigma\in\mathcal{D}_n'$. The tangent space of $GL_n$ at $E$ (resp. $F$) is denoted by $T_EGL_n$ (resp. $T_FGL_n$) and the tangent space of $\mathcal{D}_n'$ at $\Sigma$ is denoted by $T_\Sigma\mathcal{D}_n'$. The perturbation of respectively $E$, $F$ and $\Sigma$ that we consider in this paper are of the following form: $E+\dot{E}$, $F+\dot{F}$ and $\Sigma+\dot{\Sigma}$, where $\dot{E}$ and $\dot{F}$ are respectively in $T_EGL_n$ and $T_FGL_n$ and $\dot{\Sigma}$ is in $T_\Sigma\mathcal{D}_n'$.\\
 As $GL_n$ is a Lie group, $\dot{E}$ and $\dot{F}$ can be written as $EX$ and $YF$ such that $X, Y$ are in the Lie algebra of $GL_n$ which is equal to $\mathbb{C}^{n\times n}$ (since this Lie algebra is $T_{I_n}GL_n$ and $GL_n$ is an open subset in $\mathbb{C}^{n\times n}$).\\
 As $\mathcal{D}_n'$ is open in $\mathcal{D}_n$ then $T_\Sigma\mathcal{D}_n'=\mathcal{D}_n$, herein $\dot{\Sigma}=S\in\mathcal{D}_n$.
 
 Finally, the perturbations of $E$, $F$ and $\Sigma$ that we consider are as follows:\\
 $E+EX$, $F+YF$ and $\Sigma+S$, such that $X$ and $Y$ are in $\mathbb{C}^{n\times n}$ and $S$ is a diagonal matrix in $\mathbb{C}^{n\times n}$.
 
For a matrix $M\in\mathbb{C}^{n\times n}$, let $\tmop{diag} (M)$ be the diagonal matrix with the same
diagonal as $M$ and let $\tmop{off} (M)$ be the matrix where the diagonal term
of $M$ are replaced by $0$. We have $M = \tmop{diag} (M) + \tmop{off} (M)$. We
say that $M$ is an off-matrix if $M = \tmop{off} (M)$. In addition, let $(\lambda_1, \dots, \lambda_n)\in\mathbb{C}^n$, $\mathrm{diag}(\lambda_1, \dots, \lambda_n)$ is the diagonal matrix in $\mathbb{C}^{n\times n}$ of diagonal entries $\lambda_1, \dots, \lambda_n$.

The superscripts $.^t$, $.^*$ and $.^{-1}$ are used respectively for the transpose, Hermitian conjugate, and the inverse matrix.
 
 We state the following lemma which will be used in some of the proofs.
 \begin{lemma}
  \label{lem-eps-u}Let $\varphi (\varepsilon, u) = \frac{\prod_{j \geqslant 0}
  (1 + u \varepsilon^{2^j}) - 1}{\varepsilon u}$. Given $\varepsilon
  \leqslant \frac{1}{2}$, $u \leqslant 1$, and $i \geqslant 0$, we have
  \begin{eqnarray}
    \prod_{j \geqslant 0} (1 + u \varepsilon^{2^{j + i}}) & \leqslant & 1 + 2
    u \varepsilon^{2^i}  \label{ineq-lem-eps}
  \end{eqnarray}
\end{lemma}

\begin{proof}
  Modulo taking $\varepsilon^{2^i}$ instead of $\varepsilon$, it suffices to
  consider the case when $i = 0$. Now $\varphi (\varepsilon, u)$ is an
  increasing function in $\varepsilon$ and $u$, since its power series
  expansion in $\varepsilon$ and $u$ admits only positive coefficients. \
  Consequently, $\varphi (\varepsilon, u) \leqslant \varphi (\frac{1}{2}, 1) = 2$.
\end{proof}

\section{ Newton-type method for the system \large {\texorpdfstring{$FE - I_n = 0$}{TEXT}}.}\label{sec-inverse-pb}
Let $f: GL_n\times GL_n \to \mathbb{C}^{n\times n},~(E,F)\mapsto FE-I_n$. We consider the following perturbations $E+EX$, $F+YF$ of respectively $E$ and $F$ where $X,~Y\in\mathbb{C}^{n\times n}$.\\
To define the Newton sequence we have to solve the linear system obtained by canceling the linear part in the Taylor expansion of $f(E+EX, F+YF)$. The same methodology will be adopted in the next sections for the other considered systems. Hereafter, we detail the computation of the Newton sequence associated to the system $FE-I_n=0$. Moreover, a sufficient condition on the initial point for the quadratic convergence of this Newton sequence will be established.\\
Let $Z = FE - I_n$. We observe that
\begin{eqnarray}
f(E+EX, F+YF)&=&(F + YF) (E + EX) - I_n  \\
             &=&  Z + (Z + I_n) X + Y (Z + I_n) + Y (Z + I_n) X.
             \label{FYF-EEX}
\end{eqnarray}
  We assume here that $Z$ is of small norm, i.e., we start from an initial point $(E_0, F_0)$ close from the solution of the system $FE-I_n=0$.\\
  Consequently, the linear system of first order terms to solve is 
  \begin{equation}\label{linear}
      Z+X+Y=0.
  \end{equation}
Hence $X = Y = -\frac{Z}{2}$ is a solution of \Cref{linear}. Moreover we get, by substituting in Equation (\ref{FYF-EEX}) $X$ and $Y$ by $-\frac{Z}{2}$,
\begin{eqnarray}
  (F + YF) (E + EX) - I_n & = & Z^2 \left( - \frac{3}{4} I_n + \frac{Z}{4}
  \right) .  \label{eq2}
\end{eqnarray}

\begin{proposition}
  Let $Z_0 = F_0 E_0 - I_n$. Define $X_0 = - \frac{Z_0}{2}$, $E_1 = E_0 (I_n + X_0)$, $F_1 = (I_n + X_0) F_0$ and $Z_1=F_1E_1-I_n$. Assume that $\| Z_0 \| \leqslant 1$. Then
  \begin{eqnarray}
    \| Z_1 \| & \leqslant & \| Z_0 \|^2  \label{eq3}
  \end{eqnarray}
\end{proposition}

\begin{proof}
  It follows easily from (\ref{eq2}).
\end{proof}

\begin{theorem}
  Let $E_0$ and $F_0$ two complex square matrices of size $n$. Let $Z_0 = F_0 E_0 - I_n$
  and assume that $\varepsilon = \| Z_0 \| < \frac{1}{2}$. The sequences
  defined for $i \geqslant 0$
  \begin{eqnarray*}
    Z_i & = & F_i E_i - I_n\\
    X_i & = & - \frac{Z_i}{2}\\
    E_{i + 1} & = & E_i (I_n + X_i)\\
    F_{i + 1} & = & (I_n + X_i) F_i
  \end{eqnarray*}
  converge quadratically towards the solution of $FE-I_n=0$. Each $E_i$, respectively $F_i$ are invertible and, if $E_{\infty}$
  and $F_{\infty}$ are respectively the limits of sequences $(E_i)_{i
  \geqslant 0}$ and $(F_i)_{i \geqslant 0}$ we have for $i \geqslant 0,$
  \begin{eqnarray*}
    \| E_i - E_{\infty} \| & \leqslant & (1 + 2 \varepsilon) 2^{- 2^{i + 1} +
    1}_{} \varepsilon \| E_0 \|,\\
    \| F_i - F_{\infty} \| & \leqslant & (1 + 2 \varepsilon) 2^{- 2^{i + 1} +
    1}_{} \varepsilon \| F_0 \| .
  \end{eqnarray*}
\end{theorem}

\begin{proof}
First, by the assumption $\|F_0E_0-I_n\|=\| Z_0 \| < \frac{1}{2}$, we have $E_0$ and $F_0$ are invertible. In fact, $E_0F_0=I_n+E_0F_0-I_n=I_n+Z_0$ is invertible when $\|Z_0\|<1$ which is the case since we suppose $\|Z_0\|<\frac{1}{2}.$

  Let us prove by induction that $\| Z_k \| \leqslant 2^{- 2^k + 1}
  \varepsilon$. Since $\varepsilon < \frac{1}{2}$, we have
  \begin{eqnarray*}
    \| Z_{k + 1} \| & \leqslant & \| Z_k \|^2 \qquad \tmop{from} \left(
    \ref{eq3} \right)\\
    & \leqslant & \varepsilon 2^{- 2^{k + 1} + 2} \varepsilon^{} \quad\\
    & \leqslant & 2^{- 2^{k + 1} + 1} \varepsilon .
  \end{eqnarray*}
  Consequently $Z_{\infty} = 0.$ Since $X_k = - \frac{Z_k}{2}$ we deduce
  \begin{eqnarray*}
    \| X_k \| & \leqslant & 2^{- 2^k} \varepsilon .
  \end{eqnarray*}
  It follows $X_{\infty} = 0$. We have
  \begin{eqnarray*}
    E_k & = & E_{k - 1} (I_n + X_{k - 1})\\
    & = & E_0 (I_n + X_0) \cdots (I_n + X_{k - 1}) .
  \end{eqnarray*}
  Denoting $W_i = \prod_{0 \leqslant k \leqslant i} (I_n + X_k)$, $W_{\infty}
  = \prod_{k \geqslant 0} (I_n + X_k)$ we compute
  \begin{eqnarray*}
    \| W_{\infty} - I_n \| & \leqslant & \prod_{k \geqslant 0} (1 + 2^{- 2^k}
    \varepsilon) - 1\\
    & \leqslant & 2 \varepsilon\qquad\text{by using \Cref{lem-eps-u}}.
  \end{eqnarray*}
  Then $W_{\infty}$ is invertible and $\| W_{\infty}^{- 1} \| \leqslant
  \dfrac{1}{1 - 2 \varepsilon}$. Let $E_{\infty} = E_0 W_{\infty}$. Hence $E_0
  = E_{\infty} W_{\infty}^{- 1}$. In the same way $F_0 = W_{\infty}^{- 1}
  F_{\infty}$. Finally, the identity $F_{\infty} E_{\infty} - I_n = 0$ permits
  to conclude that $E_0$ and $F_0$ are invertible. In the same way we prove
  easily that $\| W_i - I_n \| \leqslant 2 \varepsilon$. It follows that $W_i$
  is invertible. Since $E_i = E_0 W_i$ we deduce that $E_i$ is invertible.
  Moreover
  \begin{eqnarray*}
    \| W_i - W_{\infty} \| & \leqslant & \| W_i \|  \left\| 1 - \prod_{k
    \geqslant i + 1} (1 + \| X_k \|) \right\|\\
    & \leqslant & (1 + \| W_i - I_n \|) \left\| \prod_{k \geqslant 0} (1 +
    2^{- 2^{k + i + 1}} \varepsilon) - 1 \right\| \\
    & \leqslant & (1 + 2 \varepsilon) 2^{- 2^{i + 1} + 1} \varepsilon \qquad\text{by using \Cref{lem-eps-u}}.
  \end{eqnarray*}
  We deduce that
  \begin{eqnarray*}
    \| E_i - E_{\infty} \| & \leqslant & (1 + 2 \varepsilon) 2^{- 2^{i + 1} +
    1} \varepsilon \| E_0 \|.
  \end{eqnarray*}
  These properties also hold for the $F_i$'s. The theorem is proved.
\end{proof}

\section{ Newton-like method for diagonalizable matrices.}\label{sec-p=1}
Let $M\in\mathcal{W}_n$, $\Sigma\in\mathcal{D}_n'$, $E,~F\in GL_n$. We aim to construct Newton sequences which converge towards the numerical solution of $f(E, F, \Sigma)=0$ where $f: GL_n\times GL_n\times\mathcal{D}_n'\to \mathbb{C}^{n\times n}\times \mathbb{C}^{n\times n},~(E, F, \Sigma)\mapsto(FE-I_n, FME-\Sigma)$.
We consider in the same way as before the perturbations $E + EX$ and $F + YF$ of respectively $E$ and $F$ and in addition the perturbation $\Sigma+S$ of $\Sigma$ such that $S\in\mathcal{D}_n$. We get with $Z=FE-I_n$ and $\Delta = FME - \Sigma$ :
\begin{align*}
&(F + YF) (E + EX) - I_n  \\
=&  Z + (Z + I_n) X + Y (Z + I_n) + Y (Z + I_n) X \numberthis \label{FE-In} \\
&(F + YF) M (E + EX) - \Sigma - S \\
= &  FME - \Sigma-S + FMEX + YFME + YFMEX\nonumber\\
= & \Delta - S + \Sigma X + Y \Sigma + \Delta X + Y \Delta + Y (\Delta +\Sigma) X \numberthis \label{FME-S} 
\end{align*}

As in the previous section we assume that $(E, F, \Sigma)$ is sufficiently close to the solution of $f(E, F, \Sigma)=0$, thus the linear system that we obtain from (\ref{FE-In}) and (\ref{FME-S}) is  
\begin{equation*}\begin{cases} Z+X+Y&=0 \\ \Delta-S+\Sigma X+Y\Sigma&=0 \end{cases}\end{equation*}

The following lemma gives a solution of this linear system.

\begin{lemma}
  \label{lem-SXY3}Let $\Sigma = \tmop{diag} (\sigma_1, \cdots, \sigma_n)$, $Z =
  (z_{i, j})_{1\le i, j\le n}$ and $\Delta = (\delta_{i, j})_{1\le i, j\le n}$ be given matrices in $\mathbb{C}^{n\times n}$. Assume that
  $\sigma_i \neq \sigma_j$ for $i \neq j$. Let $S$, $X$ and $Y$ be matrices
  defined by
  \begin{eqnarray}
    S & = & \tmop{diag} (\Delta - Z \Sigma)  \label{SXY-1}\\
    x_{i, i} & = & 0 \\
    x_{i, j} & = & \frac{- \delta_{i, j} + z_{i, j} \sigma_j}{\sigma_i -
    \sigma_j}, \qquad i \neq j \\
    y_{i, i} & = & - z_{i, i} \\
    y_{i, j} & = & \frac{\delta_{i, j} - z_{i, j} \sigma_i}{\sigma_i -
    \sigma_j}, \qquad i \neq j.  \label{SXY-5}
  \end{eqnarray}
  Then we have
  \begin{eqnarray}
    Z + X + Y & = & 0  \label{Z+X+Y=0}\\
    \Delta - S + \Sigma X + Y \Sigma & = & 0  \label{Delta-S-etc3}
  \end{eqnarray}
  Moreover
  \begin{eqnarray}
    \| X \|, \| Y \| & \leqslant & \kappa \varepsilon (K + 1)
    \label{bnd-NX-NY}
  \end{eqnarray}
  where $\varepsilon \geqslant \max (\| Z \|, \| \Delta \|)$, $\kappa = \max
  \left( 1, \max_{i \neq j} \dfrac{1}{\abs{ \sigma_i - \sigma_j }} \right)$\quad
  and $K =$\\$ \max (1, \nobracket \max_i  \abs{ \sigma_i })$.
\end{lemma}

\begin{proof}
  It is easy to verify that $X + Y + Z = 0.$ In this way the equation
  (\ref{Delta-S-etc3}) is equivalent to
  \begin{eqnarray*}
    \Delta - S - Z \Sigma + \Sigma X - X \Sigma & = & 0.
  \end{eqnarray*}
  Since $\tmop{diag} (\Delta - S - Z \Sigma) = \tmop{diag} (\Sigma X - X
  \Sigma) = 0$ the formulas which define $X$ follow easily. The bounds
  (\ref{bnd-NX-NY}) also are obvious to establish.
\end{proof}
In the next theorem we introduce the Newton sequences associated to the system $f(E, F, \Sigma)=0$ with a sufficient condition on the initial point for its quadratic convergence. 
\begin{theorem}\label{theo2}
  Let $E_0, F_0\in GL_n$ and $\Sigma_0\in\mathcal{D}_n'$ be given such that they define the sequences for $i
  \geqslant 0$,
  \begin{eqnarray*}
    Z_i & = & F_i E_i - I_n\\
    \Delta_i & = & F_i ME_i - \Sigma_i\\
    S_i & = & \tmop{diag} (\Delta_i - Z_i \Sigma_i)\\
    E_{i + 1} & = & E_i (I_n + X_i)\\
    F_{i + 1} & = & (I_n + Y_i) F_i\\
    \Sigma_{i + 1} & = & \Sigma_i + S_i,
  \end{eqnarray*}
  where $S_i$, $X_i$ and $Y_i$ are defined by the formulas
  (\ref{SXY-1}--\ref{SXY-5}). Let us define 
   $\kappa_0 = \max \left( 1, \max_{i \neq j} \dfrac{1}{\abs{
  \sigma_{0, i} - \sigma_{0, j} }} \right)$,\quad  $K_0 = \max (1,
  \nobracket \max_i  \abs{ \sigma_{0, i} })$ 
  and $\varepsilon_0= \max (\kappa_0^2 K_0^2\| Z_0 \|,
  \kappa_0^2K_0\| \Delta_0 \|)$. Assume that
  \begin{eqnarray}
  \varepsilon_0 & \leqslant & 0.033.
    \label{cond-convergence}
  \end{eqnarray}
  
  Then the sequences $(\Sigma_{i,} E_i, F_i)_{i \geqslant 0} $converge
  quadratically to the solution of $(FE - I_n, FME - \Sigma) = 0$. More precisely
  $E_0$ and $F_0$ are invertible and
  
  \begin{eqnarray*}
    \| E_i - E_{\infty} \| & \leqslant &8.1 \times 2^{1 - 2^{i + 1}} \| E_0
    \|\,\frac{\varepsilon_0}{\kappa K}
    \\
    \| F_i - F_{\infty} \| & \leqslant & 8.1 \times 2^{1 - 2^{i + 1}} \| F_0
    \| \,\frac{\varepsilon_0}{\kappa K}.
    \\
     \|\Sigma_i-\Sigma_\infty\|&\leqslant&  1.85\times 2^{1-2^{i}}
    \,\frac{\varepsilon_0}{\kappa^2 K}.
  \end{eqnarray*}
  
\end{theorem}
\begin{proof}
  Let us denote \ for each~$i \geqslant 0$,
  \[ \begin{array}{rclcrcl}
       \varepsilon_{} & = & \varepsilon_0 &  & \varepsilon_i & = & \max (\kappa_i^2 K_i^2\|
       Z_i \|,\kappa_i^2K_i \| \Delta_i \|)\\
       \kappa_{} & = & \kappa_0 & \qquad & \kappa_i & = & \max \left( 1, \;
       \max_{1 \leqslant j < k \leqslant n}  \frac{1}{\abs{ \nobracket \sigma_{i,
       k^{}}^{} - \sigma_{i, j}^{} }\nobracket} \right)\\
       K_{} & = & K_0 &  & K_i & = & \max_{1\le k\le n} \left( 1, \; \abs{ \nobracket
       \sigma_{i, k}^{} }\nobracket \right),
     \end{array} \]
  where $\sigma_{i, 1}^{}, \ldots, \sigma_{i, n}^{}$ denote the diagonal
  entries of $\Sigma_i^{}$. Let us show by induction on $i$ that
  \begin{eqnarray}
    \varepsilon_i & \leqslant & 2^{1 - 2^i} \varepsilon  \label{main-ind-13}\\
    \| \Sigma_i - \Sigma_0 \| & \leqslant & (2 - 2^{2 - 2^i})\frac{2a}{\kappa} \varepsilon
    \label{main-ind-23}
  \end{eqnarray}
with $\dis a=\frac{1}{1-8\varepsilon}$.
  These inequalities clearly hold for $i = 0$. Assuming that the induction
  hypothesis holds for a given $i$ and let us prove it for $i + 1$.
We first prove that $\dis \| \Sigma_{i+1} - \Sigma_0 \|  \leqslant  (2 - 2^{2 - 2^{i+1}}) \frac{2a}{\kappa}\varepsilon$ under the assumption $\dis \| \Sigma_{i} - \Sigma_0 \|  \leqslant  (2 - 2^{2 - 2^i})\frac{2a}{\kappa} \varepsilon$. To do this, at the first step we show that this implies $\dis K-\frac{4a}{\kappa}\varepsilon\leqslant K_i\le K+\frac{4a}{\kappa}\varepsilon$ and 
$\dis \frac{1}{1+8a\varepsilon}\kappa\leqslant\kappa_i\le \frac{\kappa}{1-8a\varepsilon}$.
 Let us prove $\dis K-\frac{4a}{\kappa}\varepsilon
\leqslant K_i\leqslant K+\frac{4a}{\kappa}\varepsilon$.
 We have
\begin{align*}
K_i:=\|\Sigma_i\| &\le \|\Sigma_0\|+\|\Sigma_i-\Sigma_0\|
\\
&\leqslant
K+(2 - 2^{2 - 2^i}) \frac{2a}{\kappa}\varepsilon
\\
&\leqslant K+\frac{4a}{\kappa}\varepsilon
 \leqslant K(1+4a\varepsilon).
\end{align*}
This implies simultaneously
$\dis K_i\geqslant K-\vert K-K_i\vert 
\geqslant K-\frac{4a}{\kappa}\varepsilon$
and $K_i\geqslant K(1-4a\varepsilon)$.
Let us show that $\dis \kappa_i\le \frac{\kappa}{1-8a\varepsilon}$. In fact, if the
 $\sigma_{i,j}$'s 
 are the diagonal values of $\Sigma_{i}^{}$, the Weyl's bound \cite{Weyl1912} implies that
 \[ \abs{ \nobracket \sigma_{i, j}  - \sigma_{0,
     j} } \nobracket \leqslant \| \Sigma_{i} - \Sigma_0 \| 
     \leqslant 
\frac{4a}{\kappa}     \varepsilon \hspace{3em} \text{for} ~1  \leqslant  j \leqslant n. \] 
  So that for $1 \leqslant j < k \leqslant n$, we obtain using 
  $1-8a\varepsilon \geqslant 0$ \ :
  \begin{eqnarray*}
    \abs{ \sigma_{i, k} - \sigma_{i, j} } & \geqslant & \abs{ \sigma_{0, k} -
    \sigma_{0, j} } - \abs{ \sigma_{i, k} - \sigma_{0, k} } - \abs{ \sigma_{i,
    j} - \sigma_{0, j} }\\
    & \geqslant & \abs{ \sigma_{0, k} - \sigma_{0, j} }
      (1 - \kappa \abs{ \sigma_{i, k} - \sigma_{0, k} }
       - \kappa \abs{ \sigma_{i, j} - \sigma_{0, j}
    })\\
    & \geqslant & \abs{ \sigma_{0, j} - \sigma_{0, k} } (1 - 8a \varepsilon)\geqslant 0.
  \end{eqnarray*}
  Finally, we get :
  \begin{eqnarray*}
    \kappa_{i}  & \leqslant & \dfrac{\kappa}{1 -8a \varepsilon} .
  \end{eqnarray*}
 On the other hand the inequality
   \begin{eqnarray*}
    \abs{ \sigma_{i, k} - \sigma_{i, j} } & \leqslant & \abs{ \sigma_{0, k} -
    \sigma_{0, j} } +\abs{ \sigma_{i, k} - \sigma_{0, k} } + \abs{ \sigma_{i,
    j} - \sigma_{0, j} }
  \end{eqnarray*}
  implies in the same way that above
  \begin{align*}
  \kappa_i&\geqslant 
  \frac{1}{1+8a\varepsilon}\kappa.
  \end{align*}
%
   Next we prove (\ref{main-ind-23}) for $i + 1$. We know $S_i=diag(\Delta_i-Z_i\Sigma_i)$. Since $\varepsilon_i=\max(\kappa_i^2K_i^2\|Z_i\|,\kappa_i^2K_i\|\Delta_i\|)$ and $\kappa_i,K_i\geqslant 1$  then $\dis \|S_i\|\le\frac{2}{\kappa_i} \varepsilon_i\leqslant
     \frac{2(1+8a\varepsilon)}{\kappa} 2^{1 - 2^i}\varepsilon$.
It follows :
  \begin{align*}
    \| \Sigma_{i + 1} - \Sigma_0 \| &\leqslant
    \|S_i \| + \| \Sigma_i - \Sigma_0 \|
    \\&\leqslant 
      \frac{2(1+8a\varepsilon)}{\kappa} 2^{1 - 2^i} \varepsilon
    + (2 - 2^{2 - 2^i}) \frac{2a}{\kappa} \varepsilon 
    \\&\leqslant
     \left (2 - 2^{1 - 2^i}\left (
     2-1
          \right )\right)  \frac{2a}{\kappa}\varepsilon
          \qquad \textrm{since $1+8a\varepsilon=a$}
     \\
    & \leqslant   \left (2 - 2^{1 - 2^i}
    \right)  \frac{2a}{\kappa}\varepsilon
  \end{align*}
  But it is eay to see that $\dis 2^{1-2^i}\geqslant 2^{2-2^{i+1}}$. Finally we get
  \begin{align*}
    \| \Sigma_{i + 1} - \Sigma_0 \| &\leqslant   \left (2 - 2^{2-2^{i+1}}
    \right)  \frac{2a}{\kappa}\varepsilon.
  \end{align*}
  Hence we can also write
  \begin{align*}
  K_i-\frac{2a}{\kappa_i}\varepsilon\leqslant \|\Sigma_i\|-\|\Sigma_{i+1}-\Sigma_i\|\leqslant
  K_{i+1}
  \leqslant \|\Sigma_i\|+\|\Sigma_{i+1}-\Sigma_i\|
  \leqslant K_i+\frac{2a}{\kappa_i}\varepsilon
  \end{align*}
  Using more the Weyl's bound we can easily get that
  \begin{align*}
  \frac{\kappa_i}{1+4a\varepsilon}\leqslant\kappa_{i+1}\leqslant
  \frac{\kappa_i}{1-4a\varepsilon}.
  \end{align*}
  Now we bound $\kappa_{i+1}^2K_{i+1}^2\|Z_{i+1}\|$. We
  have
  \begin{eqnarray*}
    Z_{i + 1} & = & Z_i X_i + Y_i Z_i + Y_i (Z_i + I_n) X_i .
  \end{eqnarray*}
Since $\dis \|X_i\|,\|Y_i\|\le \kappa_i(\|\Delta_i\|+K_i\|Z_i\|)\leqslant \frac{2}{\kappa_i K_i}\varepsilon_i$, we can write
  \begin{eqnarray*}
   \kappa_{i+1}^2K_{i+1}^2 \| Z_{i + 1} \| & \leqslant & 
   \frac{\kappa_{i+1}^2K_{i+1}^2}{\kappa_i^3K_i^3}4\varepsilon_i^2
   +\frac{\kappa_{i+1}^2K_{i+1}^2}{\kappa_i^4K_i^4}4\varepsilon_i^3
   +\frac{\kappa_{i+1}^2K_{i+1}^2}{\kappa_i^2K_i^2}4\varepsilon_i^2
   \\
    & \leqslant &
    4 \left(2+\varepsilon_i\right)
    \left(\frac{\kappa_{i+1}K_{i+1}}{\kappa_iK_i}\right)^2\varepsilon_i^2
    \\
    & \leqslant & 4\left(2+\varepsilon_i\right)
    \left(\frac{1+2a\varepsilon}{1-4a\varepsilon}\right)^2\varepsilon_i^2 %
  \end{eqnarray*}
  On the other hand
  \begin{eqnarray*}
    \Delta_{i + 1} & = & \Delta_i X_i + Y_i \Delta_i + Y_i (\Delta_i +
    \Sigma_i) X_i .
  \end{eqnarray*}
Hence
  \begin{eqnarray*}
   \kappa_{i+1}^2K_{i+1} \| \Delta_{i + 1} \| & \leqslant &
   \frac{\kappa_{i+1}^2K_{i+1}}{\kappa_i^3K_i^2}4\varepsilon_i^2
   +\frac{\kappa_{i+1}^2K_{i+1}}{\kappa_i^4K_i^3}4\varepsilon_i^3
   +\frac{\kappa_{i+1}^2K_{i+1}}{\kappa_i^2K_i}4\varepsilon_i^2   
   \\
    & \leqslant & 
     4 \left(2+\varepsilon_i\right)
    \frac{\kappa_{i+1}^2K_{i+1}}{\kappa_i^2K_i}\varepsilon_i^2
    \\
    & \leqslant & 4\left(2+\varepsilon_i\right)
   \frac{1+2a\varepsilon}{(1-4a\varepsilon)^2}\varepsilon_i^2
  \end{eqnarray*}
  It follows
\begin{eqnarray*}
    \varepsilon_{i + 1} & \leqslant & 
 4\left(2+\varepsilon\right)
    \left(\frac{1+2a\varepsilon}{1-4a\varepsilon}\right)^2\varepsilon_i^2    
    \\
    & \leqslant &    
     8\left(2+\varepsilon\right)\left(\frac{1-6\varepsilon}{1-12\varepsilon}\right)^2\varepsilon \, 2^{1-2^{i+1}}\varepsilon
    \\
    & \leqslant &  2^{1-2^{i+1}}\varepsilon
    \qquad \textrm{since  $\dis 8\left(2+\varepsilon\right)\left(\frac{1-6\varepsilon}{1-12\varepsilon}\right)^2\varepsilon\leqslant 1$ for $\varepsilon\leqslant 0.033$.}    
  \end{eqnarray*}
  This completes the proof of the two induction
  hypothesis~(\ref{main-ind-13}--\ref{main-ind-23}) at order $i + 1$.
  Let $W_i = \prod_{k = 0}^i (I_n + X_k)$. Since
  \begin{eqnarray*}
    \| X_k \| & \leqslant & \frac{2}{\kappa_kK_k} \varepsilon_k
    \\
    & \leqslant &   
     \frac{2(1+8a\varepsilon)}{\kappa K (1-4a\varepsilon)} \varepsilon_{} 2^{1 - 2^k}
     \\
    & \leqslant &  
     \frac{2}{\kappa K (1-12\varepsilon)} \varepsilon_{} 2^{1 - 2^k}
  \end{eqnarray*}
  Consequently,
  \begin{eqnarray*}
    \| W_{\infty} - I_n \| & \leqslant & \prod_{i \geqslant 0} (1 + \frac{2}{\kappa K (1-12\varepsilon)}\varepsilon
    2^{1 - 2^i}) - 1\\
    & \leqslant & \frac{4}{\kappa K (1-12\varepsilon)}\varepsilon \quad \text{from \Cref{lem-eps-u}}\\
    & \leqslant & \frac{0.22}{\kappa K}\qquad \textrm{since   $\varepsilon\leqslant 0.033$.}.
  \end{eqnarray*}
   Hence $W_{\infty}$ is invertible and $E_0 = E_{\infty} W_{\infty}^{- 1}$.
  This implies that $E_0$ is invertible. Moreover,
  \begin{eqnarray*}
    \| W_i - W_{\infty} \| & \leqslant & \| W_i \|  \left\| 1 - \prod_{k
    \geqslant i + 1} (1 + \| X_k \|) \right\|\\
    & \leqslant & (1 + \| W_i - I_n \|) \left\| \prod_{k \geqslant 0}^{} (1 +
 \frac{2}{\kappa K (1-12\varepsilon)} \varepsilon    \times 2^{1 -  2^{k + i + 1}}) - 1 \right\|\\
    & \leqslant & (1 + 0.22) \times  \frac{4}{\kappa K (1-12\varepsilon)} \times 2^{1 - 2^{i + 1}} \varepsilon\quad
     \text{from \Cref{lem-eps-u}}\\
    & \leqslant & \frac{8.1}{\kappa K} \times 2^{1 - 2^{i + 1}} \varepsilon.
  \end{eqnarray*}
   We deduce that
  \begin{eqnarray*}
    \| E_i - E_{\infty} \| & \leqslant &  \frac{8.1}{\kappa K} \times 2^{1 - 2^{i + 1}} \| E_0
    \| \varepsilon.
  \end{eqnarray*}
  In the same way we show that $F_0$ is invertible and
  \begin{eqnarray*}
    \| F_i - F_{\infty} \| & \leqslant &  \frac{8.1}{\kappa K} \times 2^{1 - 2^{i + 1}} \| F_0
    \| \varepsilon.
  \end{eqnarray*}
  Finally
  \begin{align*}
  \|\Sigma_i-\Sigma_\infty\|&\leqslant \sum_{k\geqslant i}\|\Sigma_{k+1}-\Sigma_k\|
  \\&\leqslant
  \sum_{k\geqslant i}\frac{2}{\kappa_k^2K_k}\varepsilon_k
   \\&\leqslant
 \left ( \sum_{k\geqslant 0}2^{-2^k}\right )2^{1-2^{i}}
 \frac{2}{\kappa^2 K (1-12\varepsilon)(1-8\varepsilon)}\varepsilon
 \\&\leqslant
 0.82\times 2.25\times 2^{1-2^{i}}\frac{\varepsilon}{\kappa K}
 \qquad \textrm{since $\dis \sum_{k\geqslant 0}2^{-2^k}\leqslant 0.82$ and $\varepsilon\leqslant 0.033$.}
 \\&
 \leqslant  1.85\times 2^{1-2^{i}}\varepsilon_0 .
  \end{align*}
  The theorem is proved.
\end{proof}
\begin{proposition}\label{complexity}
  The complexity of one Newton iteration in \Cref{theo2} is in $\mathcal O(n^3)$.
\end{proposition}
\begin{proof}
The computation of all the entries $x_{i,j}$, $y_{i,j}$ of $X_i$ and $Y_i$ by the formulas  (\ref{SXY-1}--\ref{SXY-5}) requires in total $\mathcal O(n^2)$ arithmetic operations. The computation of $Z_i, \Delta_i, S_i, E_{i+1}, F_{i+1}$, which requires $6$ backward stable matrix multiplications and diagonal matrix operations, has a complexity in $\mathcal O(n^3)$.
Consequently, the complexity of each iteration is in $\mathcal O(n^3)$.
\end{proof}
\begin{remark}
  It is possible to generalize this approach to the case where the diagonal matrices are
  replaced by Jordan matrices.
\end{remark}

\section{Newton-like method for two simultaneously diagonalizable matrices.} \label{sec-p=2}
Let $M_1, M_2$ be two commuting matrices in $\mathcal{W}_n$, thus $M_1$ and $M_2$ are simultaneously diagonalizable. We aim to find $E, F\in GL_n$ which diagonalize simultaneously $M_1, M_2$ so that: $FM_kE=\Sigma_k\mid k \in\{1, 2\},~\text{and}~\Sigma_1, \Sigma_2\in\mathcal{D}_n'$. This equivalent to find the numerical solution of $f(E, F, \Sigma_1, \Sigma_2)=0$ such that $f:(E, F, \Sigma_1, \Sigma_2)\mapsto(FM_1E-\Sigma_1, FM_1E-\Sigma_1)$

We consider as before the perturbations $E+EX$, $F+ YF$ and
$\Sigma_k + S_k$ of respectively $E$, $F$ and $\Sigma_k$ for $k \text{$\in \{
1, 2 \}$}$ . Letting $Z_k = \tmop{FM}_k E - \Sigma_k$ for $k = 1, 2$, we have:
\begin{align*}
  &(F + \tmop{YF}) M_k (E + \tmop{EX}) - (\Sigma_k+S_k) \\
  &= Z_k - S_k + \Sigma_k
  X + Y \Sigma_k + Z_k X + \tmop{YZ}_k + Y (Z_k + \Sigma_k) X \numberthis \label{k=1,2}
\end{align*}
By assuming $Z_1, Z_2$ are of small norm, the linear system to solve from Equation (\ref{k=1,2}) is the following 
\begin{eqnarray}
    Z_k - S_k + \Sigma_k X + Y \Sigma_k & = & 0, \qquad k = 1, 2
    \label{Delta}
  \end{eqnarray}
  A solution of (\ref{Delta}) is given by the following lemma.
\begin{lemma}
  \label{lem-SXY1} Let $\Sigma_k = \tmop{diag} (\sigma_1^k, \cdots,
  \sigma_n^k)$, $Z_k = (z^k_{i, j})_{1\le i, j\le n}$ be given matrices in $\mathbb{C}^{n\times n}$ for $k\in\{1, 2\}$. Assume that
  $\begin{vmatrix}
    \sigma_j^1  & \sigma_j^2\\
    \sigma_i^1 & \sigma_i^2
  \end{vmatrix} \neq 0$ for $i \neq j$. Let $X$, $Y$, and $S_k$ be
  the matrices defined by
  \begin{eqnarray}
    x_{i, i} & = & 0  \label{eq:23}\\
    x_{i, j} & = & \frac{\begin{vmatrix}
      \sigma_j^1 & z_{i, j}^1\\
      \sigma_j^2 & z_{i, j}^2
    \end{vmatrix}}{\begin{vmatrix}
      \sigma_i^1 & \sigma_j^1\\
      \sigma_i^2 & \sigma_j^2
    \end{vmatrix}}, \qquad i \neq j  \label{eq:24}\\
    y_{i, i} & = & 0 \\
    y_{i, j} & = & - \frac{\begin{vmatrix}
      \sigma_i^1 & z_{i, j}^1\\
      \sigma_i^2 & z_{i, j}^2
    \end{vmatrix}}{\begin{vmatrix}
      \sigma_i^1 & \sigma_j^1\\
      \sigma_i^2 & \sigma_j^2
    \end{vmatrix}}, \qquad i \neq j  \label{eq:26}\\
    S_k & = & \tmop{diag} (Z_k), \qquad k = 1, 2.  \label{SXY-11}
  \end{eqnarray}
  Then we have
  \begin{eqnarray}
    Z_k - S_k + \Sigma_k X + Y \Sigma_k & = & 0, \qquad k = 1, 2
    \label{Delta-S-etc4}
  \end{eqnarray}
  Moreover
  \begin{eqnarray}
    \| X \|, \| Y \| & \leqslant & 2 \kappa \varepsilon K  \label{eq:bnd-29}
  \end{eqnarray}
  where $\varepsilon = \max (\| Z_1 \|, \| Z_2 \|)$, $\kappa = \max \left( 1,
  \max_{i \neq j} \LARGE{\normalsize{\dfrac{1}{\begin{vmatrix}
    \sigma_i^1 & \sigma_j^1\\
    \sigma_i^2 & \sigma_j^2
  \end{vmatrix}} }} \right)$, $K =$\\$ \max (1, \nobracket \max_{i, k}
  \abs{ \sigma_i^k })$.
\end{lemma}

\begin{proof}
  It is easy to verify that the equation (\ref{Delta-S-etc4}) implies that for
  $i \neq j$,
  \begin{eqnarray*}
    \sigma_i^k x_{i, j} + \sigma_j^k y_{i, j} + z_{i, j^{}}^k & = & 0
  \end{eqnarray*}
  and that the solution of these equations is given by the formula
  {\eqref{eq:24}}, {\eqref{eq:26}}. Choosing $x_{i, i} = y_{i, i}$=0, we take
  $S_k = \tmop{diag} (Z_k + \Sigma_k X + Y \Sigma_k) = \tmop{diag} (Z_k)$
  since $\Sigma_k X + Y \Sigma_k$ is an off-matrix, to satisfy the equation
  (\ref{Delta-S-etc4}). The bounds (\ref{eq:bnd-29}) follows easily
  from {\eqref{eq:24}}, {\eqref{eq:26}}.
\end{proof}

\begin{theorem}\label{theo3}
   \label{th-quad-conv}Let $E_0$, $F_0\in GL_n$ and $\Sigma_{0, k} = \tmop{diag}
  (\sigma_{0, 1}^k, \ldots, \sigma_{0, n}^k)\in\mathcal{D}_n'$, $k = 1, 2$, be given and let
  define the sequences for $i \geqslant 0$ and $k = 1, 2$ by:
  \begin{eqnarray*}
    Z_{i, k} & = & F_i M_k E_i - \Sigma_{i, k} \quad\\
    S_{i, k} & = & \tmop{diag} (Z_{i, k})\\
    E_{i + 1} & = & E_i (I_n + X_i)\\
    F_{i + 1} & = & (I_n + Y_i) F_i\\
    \Sigma_{i + 1, k} & = & \Sigma_{i, k} + S_{i, k},
  \end{eqnarray*}
  where \ $X_i$, $Y_i$ are defined by the formulas (\ref{eq:23}--\ref{eq:26}).
  Let \normalsize $\varepsilon_0 =\max (\| Z_{0, 1} \|, \| Z_{0, 2} \|)$, $\kappa_0 =
  \max \left( 1, \max_{i \neq j} \LARGE{
  \normalsize{\dfrac{1}{\begin{vmatrix}
    \sigma_{0, i}^1&\sigma_{0, j}^1\\
    \sigma_{0, i}^2&\sigma_{0, j}^2
  \end{vmatrix}}}}  \right)$ and $K_0 = \max (1, \nobracket
  \max_{j, k}  \abs{ \sigma_{0, j}^k })$. Assume that
  \begin{eqnarray}\label{u-cond}
    u \assign 4 \varepsilon_0 \kappa_0^2 K_0^3 & \leqslant & 0.094.
  \end{eqnarray}
  Then the sequences $(\Sigma_{i, k,} E_i, F_i)_{i \geqslant 0} $converge
  quadratically to the solution of $FM_k E - \Sigma_k$ for $k = 1, 2$. More
  precisely $E_0$ and $F_0$ are invertible and
  \normalsize{\begin{eqnarray*}
    \| E_i - E_{\infty} \| & \leqslant & 1.46 \times 2^{1 - 2^{i + 1}} \| E_0
    \| u\\
    \| F_i - F_{\infty} \| & \leqslant & 1.46 \times 2^{1 - 2^{i + 1}} \| F_0
    \| u.
  \end{eqnarray*}}
\end{theorem}
\begin{proof}
  Let us denote \ for each~$i \geqslant 0$,
  \[ \begin{array}{rclcrcl}
       \varepsilon_{} & = & \varepsilon_0 &  & \varepsilon_i & = & \max (\|
       Z_{i, 1} \|, \| Z_{i, 2} \|)\\
       \kappa_{} & = & \kappa_0 & \qquad & \kappa_i & = & \max \left( 1, \;
       \max_{1 \leqslant j < k \leqslant n}  \LARGE{
       \normalsize{\dfrac{1}{\begin{vmatrix}
         \sigma_{i, j}^1 & \sigma_{i, k}^1\\
         \sigma_{i, j}^2 & \sigma_{i, k}^2
       \end{vmatrix}}}} \right)\\
       K & = & K_0 &  & K_i & = & \max (1, \nobracket \max_{j, k} (\abs{
       \nobracket \sigma_{i, j}^k } \nobracket)) \nobracket,
     \end{array} \]
  where $\sigma_{i, 1}^k, \ldots, \sigma_{i, n}^k$ are the diagonal entries
  of $\Sigma_{i, k}^{}$. Let us show by induction on $i$ that
  \normalsize{\begin{eqnarray}
    \varepsilon_i & \leqslant & 2^{1 - 2^i} \varepsilon  \label{main-ind-14}
    \\
    \| \Sigma_{i, k} - \Sigma_{0, k} \| & \leqslant & (2 - 2^{2 - 2^i})
    \varepsilon  \label{main-ind-24}
  \end{eqnarray}}
  These inequalities clearly hold for $i = 0$. Assuming that the induction
  hypothesis holds for a given $i$ and let us prove it for $i + 1$. We can notice that $\varepsilon_i \leq 1$. In fact by  induction hypothesis, we have $\varepsilon_i\leq 2^{1-2^i}\varepsilon_0$ and from (\ref{u-cond}) $\varepsilon_0=\frac{u}{4\kappa_0^2K_0^3}\leq 1$, since $u\leq 1$ and $\kappa_0, K_0 \geq 1$. As $2^{1-2^i}\leq 1,~\forall i\geq 0$, we have $\varepsilon_i\leq 1$. We first prove that $\dis \| \Sigma_{i+1,k} - \Sigma_{0,k} \|  \leqslant  (2 - 2^{2 - 2^{i+1}}) \varepsilon$ under the assumption $\dis \| \Sigma_{i,k} - \Sigma_{0,k} \|  \leqslant  (2 - 2^{2 - 2^i})\varepsilon$. To do this, at the first step we show that this implies 
$ K_i\le K+2\varepsilon$ and 
$\dis \kappa_i\le \frac{\kappa}{1-8\kappa\varepsilon(K+\varepsilon)}$.
 Let us prove 
$\dis K_i\leqslant K+2\varepsilon$.
 We have
\begin{align*}
K_i:=\|\Sigma_i\| &\le \|\Sigma_0\|+\|\Sigma_i-\Sigma_0\|
\\
&\leqslant
K+(2 - 2^{2 - 2^i}) \varepsilon
\\
&
 \leqslant K+2\varepsilon.
\end{align*}
Let us show that $\dis \kappa_i\le \frac{\kappa}{1-8\kappa\varepsilon(K+\varepsilon)}$. In fact, if the
 $\sigma_{i,j^k}$'s 
 are the diagonal values of $\Sigma_{i}^{k}$, we have
 $\abs{ \nobracket \sigma_{i , j}^k - \sigma_{0, j}^k } \nobracket \leqslant \| \Sigma_{i , k} - \Sigma_{0, k}
  \| \leqslant 2 \varepsilon $ for $1 \leqslant j \leqslant n$ and $k = 1, 2$. It follows :
  \begin{eqnarray*}
    \abs{ \nobracket \sigma_{i , j}^1 \sigma_{i , k}^2 - \sigma_{0, j}^1
    \sigma_{0, k^{}}^2 } \nobracket & = & \abs{ \nobracket \sigma_{i , j}^1
    \sigma_{i , k}^2 - \sigma_{0, j}^1 \sigma_{i , k^{}}^2 + \sigma_{0,
    j}^1 \sigma_{i , k}^2 - \sigma_{0, j}^1 \sigma_{0, k}^2 } \nobracket
    \text{}\\
    & = & \abs{ \sigma_{i , k}^2 (\sigma_{i , j}^1 - \sigma_{0, j}^1) +
    \sigma_{0, j^{}}^1 (\nobracket \sigma_{i , k}^2 - \sigma_{0, k}^2) }
    \nobracket\\
    & \leqslant & 2 \varepsilon \abs{ \sigma_{i , k}^2 } + 2 \varepsilon \abs{
    \sigma_{0, j^{}}^1 }\\
    & \leqslant & 2 \varepsilon (K + 2 \varepsilon) + 2 \varepsilon K = 4
    \varepsilon (K + \varepsilon) .
  \end{eqnarray*}
  \qquad Now,
  \begin{eqnarray*}
  \abs{\sigma_{i , j}^1 \sigma_{i , k}^2 - \sigma_{i , k}^1 \sigma_{i
    + 1, j}^2 }&\geqslant& \\
    \abs{ \sigma_{0, j}^1 \sigma_{0, k}^2 - \sigma_{0,
    k}^1 \sigma_{0, j}^2 } - \abs{ \sigma_{0, j}^1 \sigma_{0, k}^2 - \sigma_{i +
    1, j}^1 \sigma_{i , k}^2 } - \abs{ \sigma_{i , k}^1 \sigma_{i , j}^2
    - \sigma_{0, k}^1 \sigma_{0, j}^2 }&\geqslant& \\
    \abs{ \sigma_{0, j}^1 \sigma_{0, k}^2 - \sigma_{0, k}^1
    \sigma_{0, j}^2 } (1 - 8 k \varepsilon (K + \varepsilon)) .
 \end{eqnarray*}
  Finally, we get :
  \begin{eqnarray*}
    \kappa_{i } & \leqslant & \dfrac{\kappa}{1 - 8 \kappa \varepsilon (K
    + \varepsilon)} .\\
    &  &
  \end{eqnarray*} 
 To prove (\ref{main-ind-24}) it is sufficient to write
 \begin{align*}
 \|\Sigma_{i+1,k}-\Sigma_{0,k}\|&\leqslant \|S_{i,k}\|+\|\Sigma_{i+1,k}-\Sigma_{0,k}\|
 \\
 &\leqslant
 \varepsilon_i+(2-2^{2-2^i})\varepsilon
 \\
 &\leqslant
( 2^{1-2^i}+2-2^{2-2^i} )\varepsilon
\leqslant  (2-2^{2-2^{i+1}})\varepsilon.
 \end{align*}
Let us prove (\ref{main-ind-14}). Since we have
  \begin{eqnarray*}
    Z_{i + 1, k} & = & Z_{i, k} X_i + Y_i Z_{i, k} + Y_i (Z_{i, k} +
    \Sigma_{i, k}) X_i .
  \end{eqnarray*}
  we deduce 
  \begin{eqnarray*}
    \| Z_{i + 1, k} \| & \leqslant & 2 \varepsilon_i^2 \kappa_i K_i + 2
    \varepsilon_i^2 \kappa_i K_i + 4 \varepsilon_i^2 \kappa_i^2 K_i^2
    (\varepsilon_i + K_i)
    \\
    \text{} & \leqslant & 4 \varepsilon_i^2 \kappa_i^2 K_i + 4 \varepsilon_i^2
    \kappa_i^2 K_i^2 (1 + K_i) \quad \text{since $\varepsilon_i \leqslant 1
    \text{and $\kappa_i \geqslant 1$}$}\\
    & \leqslant & 3 \times 4 \varepsilon_i^2 \kappa_i^2 K_i^3 = 12
    \varepsilon_i^2 \kappa_i^2 K_i^3 \qquad \text{since $K_i \geqslant 1.$}
    \hspace{3em}
  \end{eqnarray*}
  It follows
  \begin{eqnarray*}
    \varepsilon_{i + 1} & \leqslant & \frac{12 \kappa_{}^2 (K + 2
    \varepsilon)^3}{(1 - 8 \kappa_{} \varepsilon (K + \varepsilon))^2}
    \varepsilon_i^2 \quad \leqslant \frac{12 \varepsilon \kappa_{}^2 (K + 2
    \varepsilon)^3}{(1 - 8 \kappa_{} \varepsilon (K + \varepsilon))^2} 2^{2 -
    2^{^{i + 1}}} \varepsilon\\
    & \leqslant & 3 \frac{\left( 1 + \frac{u}{2} \right)^3}{\left( 1 - 2 u
    \left( 1 + \frac{u}{4} \right) \right)^2_{}} u 2^{2 - 2^{^{i + 1}}}
    \varepsilon \quad \tmop{since} \quad \frac{\varepsilon}{K} \leqslant
    \frac{u}{4}, \kappa \varepsilon \leqslant \frac{u}{4}\\
    & \leqslant & 2^{1 - 2^{i + 1}} \varepsilon \quad \tmop{since} \quad 3
    \frac{\left( 1 + \frac{u}{2} \right)^3}{\left( 1 - 2 u \left( 1 +
    \frac{u}{4} \right) \right)^2_{}} \leqslant 2^{- 1} \tmop{for} u \leqslant
    0.094.
  \end{eqnarray*}  
  Let $W_i = \prod_{k = 0}^i (I_n + X_k)$. Since
  \begin{eqnarray*}
    \| X_l \| & \leqslant & 2 \kappa_l K_l \varepsilon_l\\
    & \leqslant & 2 \frac{\kappa}{1 - 8 \kappa \varepsilon (K + \varepsilon)}
    (K + 2 \varepsilon) \varepsilon_{} 2^{1 - 2^l}\\
    & \leqslant & \frac{\left( 1 + \frac{u}{2} \right) u}{2 \left( 1 - 2 u
    \left( 1 + \frac{u}{4} \right) \right)} 2^{1 - 2^l}\\
    & \leqslant & 0.65 \times 2^{1 - 2^l} u \quad \tmop{since} u \leqslant
    0.094.
  \end{eqnarray*}
  Consequently,
  \begin{eqnarray*}
    \| W_{\infty} - I_n \| & \leqslant & \prod_{i \geqslant 0} (1 + 0.65
    \times 2^{1 - 2^i} u) - 1\\
    & \leqslant & 1.3 u \quad \text{from \Cref{lem-eps-u}}\\
    & \leqslant & 1.3 \times 0.094 = 0.1222
  \end{eqnarray*}
  Hence $W_{\infty}$ is invertible and $E_0 = E_{\infty} W_{\infty}^{- 1}$.
  This implies that $E_0$ is invertible. Moreover,
  \begin{eqnarray*}
    \| W_i - W_{\infty} \| & \leqslant & \| W_i \|  \left\| 1 - \prod_{k
    \geqslant i + 1} (1 + \| X_k \|) \right\|\\
    & \leqslant & (1 + \| W_i - I_n \|) \left\| \prod_{k \geqslant 0}^{} (1 +
    0.059 \times 2^{1 - 2^{k + i + 1}}) - 1 \right\|\\
    & \leqslant & (1 + 0.1222) \times 1.3 \times 2^{1 - 2^{i + 1}} u\\
    & \leqslant & 1.46 \times 2^{1 - 2^{i + 1}} u.
  \end{eqnarray*}
  We deduce that
  \begin{eqnarray*}
    \| E_i - E_{\infty} \| & \leqslant & 1.46 \times 2^{1 - 2^{i + 1}} \| E_0
    \| u.
  \end{eqnarray*}
  In the same way we show that $F_0$ is invertible and
  \begin{eqnarray*}
    \| F_i - F_{\infty} \| & \leqslant & 1.46 \times 2^{1 - 2^{i + 1}} \| F_0
    \| u.
  \end{eqnarray*}
  The theorem is proved.
\end{proof}
\section{Convergence of a pencil of simultaneously diagonalizable matrices.}\label{sec-cvg-family}

In this section we present two strategies to solve the system (\ref{eq1}) of a pencil of commuting matrices $(M_i)_{1\le i\le p}$ in $\mathcal{W}_n$. The first strategy is trivial and consists of finding the common diagonalizers $E$ and $F$ of the pencil by numerically solving one of the systems $(FE - I_n, FM_1 E - \Sigma_1) = 0$ or $(FM_1 E - \Sigma_1, FM_2 E - \Sigma_1) = 0$ using \Cref{theo2} or \Cref{theo3}. Next we deduce the remaining diagonal matrices $\Sigma_i$ using the formulas
\begin{eqnarray*}
  \Sigma_{i, k} & = & \frac{E (:, k)^{\ast} M_i E (:, k)}{E (:, k)^{\ast} E
  (:, k)} \qquad 1 \leqslant k \leqslant n, \quad 2~\text{or}~3 \leqslant i \leqslant p,
\end{eqnarray*}
where $E (:, k)$ is the \emph{k-th} column in $E$.\\ In this strategy we use that a diagonalizer of one or two matrices of the pencil can diagonalize the other matrices of the pencil. We note that, in general, we don't have this property for simultaneously diagonalizable matrices, where, for instance, it is posssible to find a diagonalizer of $M_1$ which is not a common diagonalizer for the other matrices of the pencil. Nevertheless, this property holds here since we suppose that the matrices $M_i$ have simple eigenvalues. 

Another strategy is to find a ``good'' linear combination of the $M_i$'s. This
is based on \Cref{good-combination} and \Cref{theo6}.

\begin{lemma}
  \label{good-combination}Let us suppose that the $M_i$ commute pairwise and
  they are linearly independent, i.e., $\sum_{i = 1}^p a_i M_i = 0 \Rightarrow
  a_i = 0, i = 1 : p$. Let $E\in GL_n$ and $\Sigma_i\in\mathcal{D}_n'$ be such that
  \begin{eqnarray*}
    E^{- 1} M_i E - \Sigma_i & = & 0, \quad i = 1 : p.
  \end{eqnarray*}
  Let $S \in \mathbb{C}^{n \times p}$
  and the column $i$ of $S$ is the diagonal of $\Sigma_i$. Let $\sigma =
  (\sigma_{1,} \ldots, \sigma_n)$ and $\Sigma = \tmop{diag} (\sigma)$. Then
  the matrix $S$ has a full rank and $\alpha = (S^{\ast} S)^{- 1} S^{\ast}
  \sigma$ satisfies
  \begin{eqnarray*}
    \sum_{i = 1}^p \alpha_i E^{- 1} M_i E - \Sigma & = & 0.
  \end{eqnarray*}
\end{lemma}

\begin{proof}
  Since the matrices $M_i$ are simultaneously diagonalizable there exists $E$ be such that
  $E^{- 1} M_i E - \Sigma_i = 0$. \ The condition
  \begin{eqnarray*}
    \sum_{i = 1}^p \alpha_i \Sigma_i - \Sigma & = & 0
  \end{eqnarray*}
  is written as $S \alpha = \sigma$ where $S \in \mathbb{C}^{n \times p}$. The
  assumption $\sum_{i = 1}^p a_i M_i = 0 \Rightarrow a_i = 0, i = 1 : p$
  implies that the matrix has a full rank. Consequently,
  \begin{eqnarray*}
    \alpha & = & (S^{\ast} S)^{- 1} S^{\ast} \sigma .
  \end{eqnarray*}
  The lemma follows.
\end{proof}

\begin{theorem}\label{theo6}
  Let $M_1, \ldots, M_p \in \mathbb{C}^{n\times n}$ be  $p$ simultaneously diagonalizable matrices and verify the
  assumption of linearly independent. Let us consider matrices $E_0$, $F_0$
  and $\Sigma_{0, i} = \tmop{diag} (F_0 ME_0)$, $i = 1 : p$. Let us define the
  matrix $S \in \mathbb{C}^{n \times p}$ in which the column $i$ is the
  diagonal of $\Sigma_{0, i}$. Let \normalsize{$\sigma =\left( 1, e^{\frac{2 i \pi}{n}},
  \ldots, e^{\frac{2 i (n - 1) \pi}{n}} \right)$}, $\Sigma = \tmop{diag}
  (\sigma)$ and $\alpha = (S^{\ast} S)^{- 1} S^{\ast} \sigma$. We consider
  the system
  \begin{eqnarray}
    \left( \begin{array}{c}
      EF - I_n\\
      FME - \Sigma
    \end{array} \right) & = & 0  \label{eq-normalisé}
  \end{eqnarray}
  where $M = \sum_{i = 1}^p \alpha_i M_i$. If
  \begin{eqnarray*}
    {n}^2\mathrm{max}(\|Z_0\|,\|\Delta_0\|) & \leqslant & 16\times 0.033
  \end{eqnarray*}
  then $(F_0, E_0, \Sigma)$ satisfies the condition (\ref{cond-convergence})
  of \Cref{theo2}.
\end{theorem}

\begin{proof}
  In this case the quantity $\kappa$ defined in the \Cref{theo2}
  is equal to
  \begin{eqnarray*}
    \kappa & = & \frac{1}{2 ~\abs{ \nobracket \sin \left( \frac{\pi}{n} \right) }
    \nobracket}\\
    & \leqslant & \frac{n}{4} \quad \tmop{since} ~\abs{ \nobracket \sin \left(
    \frac{\pi}{n} \right) } \nobracket \geqslant \frac{2}{n} ~\tmop{for} ~n
    \geqslant 2.
  \end{eqnarray*}
  Since $K_0 = 1$ we get
  \begin{eqnarray*}
    \varepsilon_0= \max (\kappa_0^2 K_0^2\| Z_0 \|,
  \kappa_0^2K_0\| \Delta_0 \|)\leq \frac{n^2}{16} \max(\|Z_0\|,\|\Delta_0\|).
  \end{eqnarray*}
  The condition $$\mathrm{max}(\|Z_0\|,\|\Delta_0\|) \leq  0.033~\frac{16}{{n}^2},$$ gives the result.
\end{proof}
\section{Numerical illustration}\label{sec-exp}
We use a \textsc{Julia} implementation of the Newton sequences in the numerical experiments. The experimentation has been done on a Dell Windows desktop with 8 GB memory and Intel 2.3 GHz CPU. We use the Julia package \texttt{ArbNumerics} for the computation in high precision. 
\subsection{Simulation}
In this section we apply the Newton iterations presented in \Cref{theo2} (resp. \Cref{theo3}) on examples of diagonalizable matrices (resp. of two simultaneously diagonalizable matrices). We validate experimentally the sufficiency of the condition established in \Cref{theo2} (resp. \Cref{theo3}) to have a quadratic sequence (\Cref{table1,table2,table5,table6}). On the other hand, as this condition is sufficient but not necessary, we show through some other examples how this Newton sequence starting from an initial point which is not verifying this condition could converge quadratically (\Cref{table3,table4,table7,table8}). We note that the the computation in the aforementioned tables is done in high precision. Nevertheless, we test also the two Newton-type sequences using machine precision (\Cref{newton-double,newton-double1}) and this to show that these sequences have the same numerical behavior of a classical Newton method, i.e., if the solution is in the neighborhood of the initial point the Newton-type iterations will converge towards this solution with a few number of iterations and the residual error obtained at the end is in double precision.

This allows us to have an heuristic estimation on the numerical dependency of the Newton sequences from this condition to converge. Furthermore, these examples reveal the possibility of achieving computation in such problem with high precision. For example, in the case of a diagonalizable matrix of simple eigenvalues, we can compute its eigenvalues using one of the solvers which works with a double precision. Then we take this point as an initial point for the Newton sequence of \Cref{theo2} in order to increase the precision. Hereafter, we give some details about the tests: \emph{Test-1} for \Cref{theo2} and \emph{Test-2} for \Cref{theo3}, considered in this section.\\

\textbf{\emph{Test-1}.} Let $\mathbb{K}=\mathbb{R}~\text{or}~\mathbb{C}$, $M=E\Sigma E^{-1}+10^{-e}A$, where $e\in\{3, 6\}$. The matrices $E$, $\Sigma$, and $A\in\mathbb{K}^{n\times n}$ are chosen randomly following standard normal distributions such that $E$ is invertible, $\Sigma$ is diagonal with $n$ different diagonal entries and $A$ is a random square matrix obeying normal distribution of size $n$ and Frobenius norm equal to 1. Since $M$ is a small perturbation of $E\Sigma E^{-1}$, more precisely $\|M-E\Sigma E^{-1}\|_{Frob}=10^{-e}$, $M$ is a diagonalizable matrix of simple eigenvalues. Herein, we apply the Newton iteration of \Cref{theo2} on $M$ with initial point $E_0=E$, $F0=E^{-1}$ and $\Sigma_0=\Sigma$. The residual error reported in this test at iteration $k$ is given by: \begin{center}
$\text{err}_{res}=\max(\|F_kE_k-I_n\|, \|F_kME_k-\Sigma_k\|).$
\end{center}

\textbf{\emph{Test-2}.} Let $\mathbb{K}=\mathbb{R}~\text{or}~\mathbb{C}$, $M_1=F^{-1}\Sigma_1E^{-1},~M_2=F^{-1}\Sigma_2E^{-1}$, where $E$, $F$, $\Sigma_1$ and $\Sigma_2\in\mathbb{K}^{n\times n}$ are randomly sampled according to standard normal distributions, such that $E$ and $F$ are invertible, $\Sigma_1$ and $\Sigma_2$ are diagonal with $n$ different diagonal entries. The Newton iteration in \Cref{theo3} is applied on $M_1$ and $M_2$ with initial point $E_0$, $F_0$, $\Sigma_{0,1}$ and $\Sigma_{0,2}$, such that these matrices are obtained by applying a small perturbation on respectively  $E$, $F$, $\Sigma_1$ and $\Sigma_2$ as follows:\\
$E_0=E+10^{-e}A$, $F_0=F+10^{-e}B$, $\Sigma_{0,1}=\Sigma_1+10^{-e}C$, $\Sigma_{0,2}=\Sigma_2+10^{-e}D$, where $e\in\{3, 6\}$, $A$ and $B$ (resp. $C$ and $D$) are random square matrices (resp. random diagonal matrices with different diagonal entries) of size $n$ and Frobenius norm equal to 1, with entries in $\mathbb{K}$ following standard normal distributions. The residual error reported in this test at iteration $k$ is given by: \begin{center}
$\text{err}_{res}=\max(\|F_kM_1E_k-\Sigma_{k,1}\|, \|F_kM_2E_k-\Sigma_{k,2}\|).$
\end{center}

We notice that the condition established in \Cref{theo2} (resp. \Cref{theo3}) is reached in \emph{Test-1} (resp. \emph{Test-2}) for matrices of size $10$ with order of perturbation equal to $10^{-6}$, and we can see in \Cref{table1,table2,table5,table6} that the Newton sequences with initial point verifying the condition in the associated theorem converge quadratically. We can notice also that by increasing the perturbation up to $10^{-3}$ (the initial point does not verify the condition in the associated theorem), the Newton sequences converge quadratically for different sizes of matrices $n=10,~50,~100$ (see \Cref{table3,table4,table7,table8}). Moreover, we can notice in \Cref{newton-double} the Newton-type iteration of \Cref{theo2} applied in double precision converges with a few number of iterations $\sim$ 5 and the final residual error measured with the Frobenius norm is of order machine precision $\sim 10^{-14}$ and it is of the same order obtained by the standard Julia method \texttt{eigen} to compute the eigen decomposition. The same remarks are valid for \Cref{newton-double1} where the Newton-type sequence of \Cref{theo3} needs, in double precision, a few iterations to converges towards the solution given by using the Frobenius norm a residual error of order machine precision.     

\begin{table}[ht!]
\centering
\caption{The computational results throughout 7 iterations of an example of implementation of \emph{Test-1} with $\mathbb{K}=\mathbb{R},~n=10~\text{and}~e=6$ in precision $1024$.}
\begin{tabular}{ |c|c|c| }
 \hline
 Iteration & $ \varepsilon:=\max (\kappa_0^2 K_0^2\| Z_0 \|,
  \kappa_0^2K_0\| \Delta_0 \|)\le0.033$&$\text{err}_{res}$  \\ \hline
 1 &0.00131&9.33$e-6$ \\
 2 &2.39$e-8$&1.06$e-10$ \\
  3 &1.68$e-18$&7.49$e-21$  \\
   4 &2.93$e-38$&1.31$e-40$\\
    5 &4.21$e-78$&1.87$e-80$ \\
    6&1.17$e-157$&5.24$e-160$\\
    7& 4.16$e-288$&6.20$e-293$ \\
 \hline
\end{tabular}

\label{table1}
\end{table}
\begin{table}[ht!]
\centering
\caption{The computational results throughout 7 iterations of an example of implementation of \emph{Test-1} with $\mathbb{K}=\mathbb{C},~n=10~\text{and}~e=6$ in precision $1024$.}
\begin{tabular}{ |c|c|c| }
 \hline
 Iteration &$ \varepsilon:=\max (\kappa_0^2 K_0^2\| Z_0 \|,
  \kappa_0^2K_0\| \Delta_0 \|)\le0.033$ &$\text{err}_{res}$  \\ \hline
 1 &0.02581 &2.76$e-4$ \\
 2 & 3.49$e-6$ & 2.33$e-8$ \\
  3 & 9.51$e-14$& 6.34$e-16$  \\
   4 &5.31$e-29$ & 3.54$e-31$\\
    5 &1.96$e-59$ & 1.31$e-61$ \\
    6&3.02$e-120$ & 2.01$e-122$\\
    7& 4.58$e-242$ & 3.05$e-244$ \\
 \hline
\end{tabular}

\label{table2}
\end{table}
\begin{table}[ht!]
\centering
\caption{The residual error throughout 7 iterations given by the implementation of \emph{Test-1} with $\mathbb{K}=\mathbb{R}, e=3$ and $n=10, 50, 100$ in precision $1024$.}
\begin{tabular}{ |c|c|c|c| }
 \hline
 Iteration & $n=10$ & $n=50$&$n=100$ \\ \hline
 1 &8.57$e-3$ &7.93$e-2$ &3.22$e-2$ \\
 2 &1.91$e-4$ &5.76$e-2$&1.38$e-2$\\
  3 & 1.58$e-8$&6.19$e-3$&6.12$e-4$\\
   4 & 4.79$e-16$&8.74$e-5$&5.42$e-7$\\
    5 & 3.56$e-31$&1.31$e-8$&3.83$e-13$\\
    6& 1.39$e-61$&2.39$e-16$&1.80$e-25$\\
    7&1.91$e-122$&7.03$e-32$&3.81$e-50$ \\
 \hline
\end{tabular}

\label{table3}
\end{table}
\begin{table}[ht!]
\centering
\caption{The residual error throughout 7 iterations given by the implementation of \emph{Test-1} with $\mathbb{K}=\mathbb{C}, e=3$ and $n=10, 50, 100$ in precision $1024$.}
\begin{tabular}{ |c|c|c|c| }
 \hline
 Iteration & $n=10$ & $n=50$&$n=100$ \\ \hline
 1 &8.84$e-3$ &9.75$e-2$&1.61$e-2$ \\
 2 &8.59$e-6$ &6.39$e-5$&1.03$e-4$\\
  3 &3.91$e-11$ &3.99$e-9$&4.68$e-9$\\
   4 &9.87$e-22$ &1.87$e-17$&3.13$e-17$\\
    5 &7.60$e-43$ &4.42$e-34$&8.84$e-34$\\
    6&5.14$e-85$ &2.50$e-67$&9.45$e-67$\\
    7&2.64$e-169$ &8.28$e-134$&1.05$e-132$ \\
 \hline
\end{tabular}

\label{table4}
\end{table}
\begin{table}[ht!]
\centering
\caption{The residual error throughout 5 iterations given by the implementation of \emph{Test-1} with $\mathbb{K}=\mathbb{R}, e=3$ and $n=10, 20, 30$, in double precision.}
\begin{tabular}{ |c|c|c|c| }
 \hline
 Iteration & $n=10$ & $n=20$&$n=30$ \\ \hline
 1  &4.78$e-3$  &1.01$e-2$ &1.01$e-2$\\
 2 &4.71$e-3$ &2.55$e-3$&1.14$e-3$ \\
  3 &2.29$e-5$ &1.97$e-5$ &4.08$e-7$ \\
   4 &1.43$e-9$ &2.36$e-10$ &2.26$e-13$ \\
    5 &4.06$e-15$ &1.23$e-14$ &5.04$e-14$ \\\hline
    $\|M-E_{\texttt{eigen}}\Sigma_{\texttt{eigen}} E^{-1}_{\texttt{eigen}}\|_{Frob}$&9.49$e-15$  &2.83$e-14$&7.45$e-14$\\\hline
    $\|M-E_{\texttt{newton}}\Sigma_{\texttt{newton}} E^{-1}_{\texttt{newton}}\|_{Frob}$&2.96$e-15$ &1.01$e-14$&3.42$e-14$ \\
 \hline
\end{tabular}

\label{newton-double}
\end{table}
\begin{table}[ht!]
\centering
\caption{The computational results throughout 7 iterations of an example of implementation of \emph{Test-2} with $\mathbb{K}=\mathbb{R},~n=10~\text{and}~e=6$ in precision $1024$.}
\begin{tabular}{ |c|c|c| }
 \hline
 Iteration & $4\kappa^2K^3\varepsilon\le0.094$&$\text{err}_{res}$  \\ \hline
 1 &7.65$e-2$&6.72$e-6$ \\
 2 &1.73$e-7$&1.52$e-11$ \\
  3 &5.58$e-18$&4.90$e-22$  \\
   4 &5.49$e-39$&4.82$e-43$\\
    5 & 3.10$e-81$& 2.73$e-85$ \\
    6& 2.28$e-165$& 2.01$e-169$\\
    7&  2.20$e-279$& 1.94$e-283$ \\
 \hline
\end{tabular}

\label{table5}
\end{table}
\begin{table}[ht!]
\centering
\caption{The computational results throughout 7 iterations of an example of implementation of \emph{Test-2} with $\mathbb{K}=\mathbb{C},~n=10~\text{and}~e=6$ in precision $1024$.}
\begin{tabular}{ |c|c|c| }
 \hline
 Iteration & $4\kappa^2K^3\varepsilon\le0.094$&$\text{err}_{res}$  \\ \hline
 1 &6.86$e-3$&9.16$e-6$ \\
 2 &7.14$e-9$&9.53$e-12$ \\
  3 &9.51$e-21$&1.26$e-23$  \\
   4 &6.69$e-44$&8.92$e-47$\\
    5 & 3.77$e-90$& 5.04$e-93$ \\
    6& 2.59$e-182$& 3.45$e-185$\\
    7&  1.65$e-281$& 2.20$e-284$ \\
 \hline
\end{tabular}

\label{table6}
\end{table}
\begin{table}[ht!]
\centering
\caption{The residual error throughout 7 iterations given by the implementation of \emph{Test-2} with $\mathbb{K}=\mathbb{R}, e=3$ and $n=10, 50, 100$ in precision $1024$.}
\begin{tabular}{ |c|c|c|c| }
 \hline
 Iteration & $n=10$ & $n=50$&$n=100$ \\ \hline
 1 &2.91$e-2$ &4.57$e-3$ &1.01$e-2$ \\
 2 &7.97$e-5$ & 1.03$e-6$&1.31$e-6$ \\
  3 &4.21$e-9$ &1.69$e-11$ &3.71$e-11$ \\
   4 &1.07$e-16$ &2.42$e-23$ &1.23$e-22$ \\
    5 &3.92$e-33$ &1.18$e-44$ &1.46$e-43$ \\
    6&2.63$e-64$ &1.02$e-89$ &1.67$e-86$ \\
    7&1.71$e-128$ &3.20$e-177$ &9.01$e-172$  \\
 \hline
\end{tabular}

\label{table7}
\vspace{0.3cm}
\caption{The residual error throughout 7 iterations given by the implementation of \emph{Test-2} with $\mathbb{K}=\mathbb{C}, e=3$ and $n=10, 50, 100$ in precision $1024$.}
\begin{tabular}{ |c|c|c|c| }
 \hline
 Iteration & $n=10$ & $n=50$&$n=100$ \\ \hline
 1 &7.33$e-3$ &3.14$e-3$ &5.52$e-3$ \\
 2 &3.49$e-6$ &7.48$e-7$ &1.35$e-6$ \\
  3 &2.91$e-12$ &1.11$e-13$ &1.19$e-13$ \\
   4 &2.04$e-24$ &2.54$e-27$ &1.68$e-27$ \\
    5 &8.23$e-49$ &3.04$e-54$ &2.19$e-54$ \\
    6&1.88$e-97$ &3.41$e-108$ &1.50$e-108$ \\
    7&1.31$e-194$ &1.91$e-215$ &4.53$e-216$  \\
 \hline
\end{tabular}

\label{table8}
 \end{table}
 \begin{table}[ht!]
\centering
\caption{The residual error throughout 5 iterations given by the implementation of \emph{Test-2} with $\mathbb{K}=\mathbb{R}, e=3$ and $n=10, 20, 30$, in double precision.}
\resizebox{\textwidth}{!}{\begin{tabular}{ |c|c|c|c| }
 \hline
 Iteration & $n=10$ & $n=20$&$n=30$ \\ \hline
 1  &2.71$e-3$ &1.21$e-2$ &4.64$e-3$\\
 2 &1.36$e-6$&4.91$e-6$ &2.24$e-6$ \\
  3 &1.39$e-12$ &2.57$e-11$ &4.74$e-11$  \\
   4 &6.16$e-15$
  &8.97$e-14$
 &1.55$e-13$ \\
    5 &7.04$e-15$ & 8.09$e-14$ &1.53$e-13$  \\\hline
    \small{$\max(\|M_1-E\Sigma_{1} E^{-1}\|_{Frob}, \|M_2-E\Sigma_{2} E^{-1}\|_{Frob})$}&3.74$e-15$ &4.13$e-14$ &8.21$e-14$ \\
 \hline
\end{tabular}}

\label{newton-double1}
\end{table}
 \newpage
 \subsection{Cauchy matrix}\label{wilkinson}
 In this section we present an example for a Cauchy matrix of size $n=13$ of entries $a_{i, j}=\frac{1}{i+j},~\forall 1\le i, j\le 13$, that illustrates how the Newton-type iteration can be used to increase the accuracy of the eigenvalues. 
 We take the eigen decomposition given by the standard \textsc{Julia} method \texttt{eigen} from the package \texttt{LinearAlgebra} as an initial point of Newton sequences in \Cref{theo2} with 5 iterations. The computation is done with the precision 1024 using \texttt{ArbNumerics} package. The initial point given by \texttt{eigen} is in double precision. It is converted to the precision 1024 using \texttt{ArbNumerics} package, in order to apply Newtons iterations with this precision of 1024 bits. 
 In \Cref{cauchy-table} we report the eigenvalues given by \texttt{eigen} ($\sigma_{\texttt{eigen}}$) and the eigenvalues rounded to the double precision given by Newton-type sequence ($\sigma_{\texttt{newton}}$) initialized with \texttt{eigen}. We also report the relative error $\abs{\sigma_{\texttt{newton}}-\sigma_{\texttt{eigen}}}\over{\sigma_{\texttt{newton}}}$ in order to show the refinement amount realized by the Newton method. As we can see the matrix of this example is ill-conditioned (Cauchy matrices are in general ill-conditioned). There is a cluster of eigenvalues nearby zero. The accuracy enhancement obtained by applying  Newton-type iterations can be clearly seen in \Cref{cauchy-table}, in particular for the first four smallest eigenvalues. For instance, the smallest eigenvalue returned by \texttt{eigen} is of order $10^{-17}$ close to the second smallest eigenvalues of order $10^{-16}$. Newton-type method shows that the smallest eigenvalue of the order $10^{-19}$ yields a large relative error $\sim 39.33$. This also shows that all the eigenvalues are well-separated.      

 \begin{table}
 \caption{The relative error between $\sigma_{\texttt{eigen}}$ from the method \texttt{eigen} and $\sigma_{\texttt{newton}}$ from the Newton-type method for the Cauchy matrix $\big(\frac{1}{i+j}\big)_{1\le i, j\le 13}$.}
\begin{tabular}{ |c|c|c|c| }
 \hline
 Eigenvalue & $\sigma_{\texttt{eigen}}$ &$\sigma_{\texttt{newton}}$& ${\abs{\sigma_{\texttt{newton}}-\sigma_{\texttt{eigen}}}}\over{\sigma_{\texttt{newton}}}$ \\ \hline
 1 &\texttt{2.4030587641505818e-17} & \texttt{5.958203769841865e-19}&39.33 \\
 2 &\texttt{1.8824087522342697e-16} & \texttt{1.7156976132548192e-16}&0.09716 \\
  3 &\texttt{2.3152722725223998e-14} &\texttt{2.3178576801522747e-14} &0.00111\\
   4 &\texttt{1.9513972147589434e-12}&\texttt{1.951356013568409e-12} &2.11$e-5$ \\
    5 &\texttt{1.1466969172503778e-10} & \texttt{1.1466967568738049e-10}&1.39$e-7$ \\
    6&\texttt{4.991788233415145e-9} &\texttt{4.991788235245136e-9} &3.66$e-10$ \\
    7&\texttt{1.6668681228080362e-7} &\texttt{1.666868122813953e-7} &3.54$e-12$  \\
    8&\texttt{4.360227301207107e-6} & \texttt{4.360227301206033e-6}&2.46$e-13$  \\
    9&\texttt{9.040674871074817e-5}& \texttt{9.040674871075823e-5}&1.11$e-13$  \\
    10&\texttt{0.0014925044272821445} & \texttt{0.0014925044272821172}&1.83$e-14$  \\
    11&\texttt{0.01955788569925287} &\texttt{0.01955788569925287} &4.81$e-17$  \\
    12&\texttt{0.19958813407010345} &\texttt{0.19958813407010337}&4.64$e-16$  \\
    13&\texttt{1.3693334145989837} &\texttt{1.3693334145989824} &9.98$e-16$  \\
 \hline
\end{tabular}

\label{cauchy-table}
 \end{table}
 
\subsection{Sub-matrix iterations}

It is possible to adapt the proposed method, taking into account the condition of the eigenvalue $\sigma_i$ given by the quantity 
\begin{align*}
  \kappa (\sigma_i) & =  \max_{i \neq j} \left( 1, \frac{1}{\vert \sigma_i -
  \sigma_j \vert} \right)
  \end{align*}
Theoretical results imply that the computation of clusters of eigenvalues is ill-conditioned.
However, one can apply Theorem 3 on sub-matrices to improve the well-conditioned eigenvalues. We denote
\begin{eqnarray*}
  \delta & = & \sqrt{\frac{K \| \Delta_0 \|}{0.033}}
\end{eqnarray*}
and $p$ the index such that $\Sigma = \left( \begin{array}{cc}
  \Sigma_p & \\
  & \Sigma_{n - p}
\end{array} \right)$, $\Sigma_p = \tmop{diag} (\sigma_1, \ldots, \sigma_p)$, $\Sigma_{n-p} = \tmop{diag} (\sigma_{p+1}, \ldots, \sigma_n)$ and 
$\vert \sigma_i - \sigma_j \vert > \delta$ for all $1 \leqslant i \le p$ and $i < j \leqslant n$. We adapt Newton iteration to the block associated with the well-conditioned eigenvalues by defining the matrices $X$, $Y$ and $S$ as follows:
  \begin{eqnarray*}
    x_{i, i} & = & 0\\
    x_{i, j} & = & \left\{ \begin{array}{cc}
      \dfrac{-\delta_{i,j}+z_{i,j}\sigma_j}{\sigma_i - \sigma j} \quad & \tmop{if} \quad \vert \sigma_i -
      \sigma_j \vert > \delta\\
      0 & \tmop{otherwise}\\
    \end{array} \right.\\
    Y &=& - Z -X\\
      S &=&  \tmop{diag} (- \Delta + Z \Sigma).
  \end{eqnarray*}
\Cref{table12} (resp. \Cref{table13}) shows the residual error $\text{err}_{res}$ as in \emph{Test-1} for the Cauchy matrix of size $200$ (resp. the Rosser matrix of size 256 \cite{Rosser}) by applying the aforementioned sequences, the initial point is given by the Julia method \texttt{eigen}. The computation is done in precision 1024. 
\begin{table}[ht!]
\centering
\caption{The residual error throughout 6 iterations with the Cauchy matrix of size $200$.}
\begin{tabular}{ |c|c|c| }
 \hline
 Iteration & $p=12$, $\delta=4.51e-7$ & $p=5$, $\delta=4.51e-7$ \\ \hline
 1  &2.45$e-15$ &2.35$e-15$ \\
 2 &9.63$e-26$&3.75$e-29$ \\
  3 &1.56$e-36$ &1.21$e-53$  \\
   4 &1.54$e-45$  &1.81$e-83$ \\
    5 &1.15$e-54$ & 3.49$e-110$   \\
     6 &5.08$e-64$ & 8.67$e-137$   \\\hline
\end{tabular}
\label{table12}
\end{table}
\begin{table}[ht!]
\centering
\caption{The residual error throughout 6 iterations with the Rosser matrix of size $256$.}
\begin{tabular}{ |c|c|c| }
 \hline
 Iteration & $p=11$, $\delta=1.11e-3$ & $p=5$, $\delta=1.11e-3$ \\ \hline
 1  &7.15$e-12$ &1.65$e-12$ \\
 2 &7.18$e-20$&7.18$e-20$ \\
  3 &1.42$e-40$ &1.81$e-41$  \\
   4 &1.73$e-53$  &1.56$e-85$ \\
    5 &7.17$e-66$ & 1.75$e-119$   \\
     6 &8.79$e-79$ &  8.11$e-153$   \\\hline
\end{tabular}
\label{table13}
\end{table}

\section{Conclusion}
Taking a Newton approach towards systems of equations describing the simultaneous diagonalization problem of diagonalizable matrices, leads us to new algorithmic insights. We exhibit a Newton-type method without solving a linear system at each step as is the case of a classical Newton method. The numerical experiments corroborate the quadratic convergence predicted by the theoretical analysis.
\\
 We focused on the regular case. Some improvements and extensions can be considered, such as the treatment of clusters of eigenvalues. Another direction that can be explored, is the construction of higher-order methods.

\section{Conflict of Interest}
The authors declare no conflicting interest that is directly or indirectly related to the work submitted for publication.
\bibliography{simdiag1}

\begin{thebibliography}{10}

\bibitem{absil1}
P.-A. Absil and K.A. Gallivan.
\newblock Joint diagonalization on the oblique manifold for independent
  component analysis.
\newblock In {\em 2006 IEEE International Conference on Acoustics Speech and
  Signal Processing Proceedings}, volume~5, pages V--V, 2006.

\bibitem{AbsMahSep2008}
P.-A. Absil, R.~Mahony, and R.~Sepulchre.
\newblock {\em Optimization Algorithms on Matrix Manifolds}.
\newblock Princeton University Press, Princeton, NJ, 2008.

\bibitem{Afsari}
B.~Afsari.
\newblock Sensitivity analysis for the problem of matrix joint diagonalization.
\newblock {\em SIAM Journal on Matrix Analysis and Applications},
  30:1148--1171, 2008.

\bibitem{lineargrp}
E.~Andruchow, G.~Larotonda, L.~Recht, and A.~Varela.
\newblock The left invariant metric in the general linear group.
\newblock {\em Journal of Geometry and Physics}, 86:241--257, 2014.

\bibitem{bouchard}
Florent Bouchard, Bijan Afsari, Jérôme Malick, and Marco Congedo.
\newblock Approximate joint diagonalization with {R}iemannian optimization on
  the general linear group.
\newblock {\em SIAM Journal on Matrix Analysis and Applications},
  41(1):152--170, 2020.

\bibitem{bro97}
Rasmus Bro.
\newblock Parafac tutorial and applications.
\newblock {\em Chemometrics and Intelligent Laboratory Systems},
  38(2):149--171, 1997.

\bibitem{BBM92}
Angelika Bunse-Gerstner, Ralph Byers, and Volker Mehrmann.
\newblock A chart of numerical methods for structured eigenvalue problems.
\newblock {\em SIAM Journal on Matrix Analysis and Applications},
  13(2):419--453, 1992.

\bibitem{BBM93}
Angelika Bunse-Gerstner, Ralph Byers, and Volker Mehrmann.
\newblock Numerical methods for simultaneous diagonalization.
\newblock {\em SIAM Journal on Matrix Analysis and Applications},
  14(4):927--949, 1993.

\bibitem{bu13}
Peter B{\"u}rgisser, Michael Clausen, and Mohammad~A Shokrollahi.
\newblock {\em Algebraic complexity theory}, volume 315.
\newblock Springer Science \& Business Media, 2013.

\bibitem{blind}
Jean-François Cardoso and Antoine Souloumiac.
\newblock Blind beamforming for non gaussian signals.
\newblock {\em IEE Proceedings-F}, 140:362--370, 1993.

\bibitem{blind1}
Jean-François Cardoso and Antoine Souloumiac.
\newblock Jacobi angles for simultaneous diagonalization.
\newblock {\em SIAM Journal on Matrix Analysis and Applications},
  17(1):161--164, 1996.

\bibitem{ca-ch}
J~Douglas Carroll and Jih-Jie Chang.
\newblock Analysis of individual differences in multidimensional scaling via an
  n-way generalization of “{E}ckart-{Y}oung” decomposition.
\newblock {\em Psychometrika}, 35(3):283--319, 1970.

\bibitem{ci-ma}
Andrzej Cichocki, Danilo Mandic, Lieven De~Lathauwer, Guoxu Zhou, Qibin Zhao,
  Cesar Caiafa, and Huy~Anh Phan.
\newblock Tensor decompositions for signal processing applications: From
  two-way to multiway component analysis.
\newblock {\em IEEE signal processing magazine}, 32(2):145--163, 2015.

\bibitem{BSS}
Pierre Comon and Christian Jutten.
\newblock {\em Handbook of Blind Source Separation: Independent Component
  Analysis and Applications}.
\newblock Academic Press, Inc., USA, 1st edition, 2010.

\bibitem{CoxUsingalgebraicgeometry2005}
David~A. Cox, John~B. Little, and Donal O'Shea.
\newblock {\em Using Algebraic Geometry}.
\newblock Number 185 in Graduate Texts in Mathematics. Springer, New York, 2nd
  edition, 2005.

\bibitem{lath06}
Lieven De~Lathauwer.
\newblock A link between the canonical decomposition in multilinear algebra and
  simultaneous matrix diagonalization.
\newblock {\em SIAM journal on Matrix Analysis and Applications},
  28(3):642--666, 2006.

\bibitem{alg2}
S.C. Douglas.
\newblock Self-stabilized gradient algorithms for blind source separation with
  orthogonality constraints.
\newblock {\em IEEE Transactions on Neural Networks}, 11(6):1490--1497, 2000.

\bibitem{elkadi_introduction_2007}
Mohamed Elkadi and Bernard Mourrain.
\newblock {\em Introduction {\`a} la r{\'e}solution des syst{\`e}mes
  polynomiaux}, volume~59 of {\em Math{\'e}matiques et Applications}.
\newblock Springer, 2007.

\bibitem{flury1994simultaneous}
Bernard~D Flury and Beat~E Neuenschwander.
\newblock Simultaneous diagonalization algorithms with applications in
  multivariate statistics.
\newblock In {\em Approximation and computation: A Festschrift in honor of
  Walter Gautschi}, pages 179--205. Springer, 1994.

\bibitem{blind3}
Bernhard~N. Flury and Walter Gautschi.
\newblock An algorithm for simultaneous orthogonal transformation of several
  positive definite symmetric matrices to nearly diagonal form.
\newblock {\em SIAM Journal on Scientific and Statistical Computing},
  7(1):169--184, 1986.

\bibitem{chatelin}
Chatelin Françoise.
\newblock Simultaneous newton's iteration for the eigenproblem.
\newblock {\em Computing}, 5:67--74, 1984.

\bibitem{Hoeven}
J.~V.~D. Hoeven and B.~Mourrain.
\newblock Efficient certification of numeric solutions to eigenproblems.
\newblock In {\em MACIS}, 2017.

\bibitem{JVDHYak}
Joris van~der Hoeven and Jean-Claude Yakoubsohn.
\newblock Certified singular value decomposition.
\newblock Technical Report HAL 01941987, 2018.

\bibitem{HJ91}
Roger~A. Horn and Charles~R. Johnson.
\newblock {\em Topics in Matrix Analysis}.
\newblock Cambridge University Presss, Cambridge, 1991.

\bibitem{HJ12}
Roger~A Horn and Charles~R. Johnson.
\newblock {\em Matrix analysis}.
\newblock Cambridge University Press, Cambridge, 2012.

\bibitem{jiang2016simultaneous}
Rujun Jiang and Duan Li.
\newblock Simultaneous diagonalization of matrices and its applications in
  quadratically constrained quadratic programming.
\newblock {\em SIAM Journal on Optimization}, 26(3):1649--1668, 2016.

\bibitem{JM}
M.~Joho and K.~Rahbar.
\newblock Joint diagonalization of correlation matrices by using {N}ewton
  methods with application to blind signal separation.
\newblock {\em Sensor Array and Multichannel Signal Processing Workshop
  Proceedings, 2002}, pages 403--407, 2002.

\bibitem{joho2008newton}
Marcel Joho.
\newblock Newton method for joint approximate diagonalization of positive
  definite hermitian matrices.
\newblock {\em SIAM Journal on Matrix Analysis and Applications},
  30(3):1205--1218, 2008.

\bibitem{joho2002joint}
Marcel Joho and Kamran Rahbar.
\newblock Joint diagonalization of correlation matrices by using newton methods
  with application to blind signal separation.
\newblock In {\em Sensor Array and Multichannel Signal Processing Workshop
  Proceedings, 2002}, pages 403--407. IEEE, 2002.

\bibitem{jolli}
Ian Jolliffe.
\newblock {\em Principal Component Analysis}, pages 1094--1096.
\newblock Springer Berlin Heidelberg, Berlin, Heidelberg, 2011.

\bibitem{laub1987computation}
Alanj Laub, MICHAELT Heath, C~Paige, and R~Ward.
\newblock Computation of system balancing transformations and other
  applications of simultaneous diagonalization algorithms.
\newblock {\em IEEE Transactions on Automatic Control}, 32(2):115--122, 1987.

\bibitem{luc-alb}
Xavier Luciani and Laurent Albera.
\newblock Canonical polyadic decomposition based on joint eigenvalue
  decomposition.
\newblock {\em Chemometrics and Intelligent Laboratory Systems}, 132:152--167,
  2014.

\bibitem{MAHONY199667}
R.E. Mahony.
\newblock The constrained newton method on a lie group and the symmetric
  eigenvalue problem.
\newblock {\em Linear Algebra and its Applications}, 248:67--89, 1996.

\bibitem{mes-bel}
Ammar Mesloub, Adel Belouchrani, and Karim Abed-Meraim.
\newblock Efficient and stable joint eigenvalue decomposition based on
  generalized givens rotations.
\newblock In {\em 2018 26th European Signal Processing Conference (EUSIPCO)},
  pages 1247--1251. IEEE, 2018.

\bibitem{alg3}
M.~Nikpour, J.~Manton, and G.~Hori.
\newblock Algorithms on the {S}tiefel manifold for joint diagonalisation.
\newblock {\em 2002 IEEE International Conference on Acoustics, Speech, and
  Signal Processing}, 2:II--1481--II--1484, 2002.

\bibitem{alg4}
Yasunori Nishimori and Shotaro Akaho.
\newblock Learning algorithms utilizing quasi-geodesic flows on the {S}tiefel
  manifold.
\newblock {\em Neurocomput.}, 67:106–135, August 2005.

\bibitem{alg1}
Kamran Rahbar and James~P. Reilly.
\newblock Geometric optimization methods for blind source separation of
  signals.
\newblock In {\em in Proc. ICA}, pages 375--380, 2000.

\bibitem{Rosser}
B.~Rosser, C.~Lanczos, M.R. Hestenes, and W.~Karush.
\newblock Separation of close eigen-values of a real symmetric matrix.
\newblock {\em Journal of Research of the National Bureauof Standards}, 47,
  1950.

\bibitem{sato2017riemannian}
Hiroyuki Sato.
\newblock Riemannian newton-type methods for joint diagonalization on the
  stiefel manifold with application to independent component analysis.
\newblock {\em Optimization}, 66(12):2211--2231, 2017.

\bibitem{si-bro}
Nicholas~D Sidiropoulos and Rasmus Bro.
\newblock On the uniqueness of multilinear decomposition of n-way arrays.
\newblock {\em Journal of chemometrics}, 14(3):229--239, 2000.

\bibitem{si-lath}
Nicholas~D Sidiropoulos, Lieven De~Lathauwer, Xiao Fu, Kejun Huang, Evangelos~E
  Papalexakis, and Christos Faloutsos.
\newblock Tensor decomposition for signal processing and machine learning.
\newblock {\em IEEE Transactions on Signal Processing}, 65(13):3551--3582,
  2017.

\bibitem{so-lath17-1}
Mikael S{\o}rensen and Lieven De~Lathauwer.
\newblock Multidimensional harmonic retrieval via coupled canonical polyadic
  decomposition—part i: Model and identifiability.
\newblock {\em IEEE Transactions on Signal Processing}, 65(2):517--527, 2016.

\bibitem{so-lath17-2}
Mikael S{\o}rensen and Lieven De~Lathauwer.
\newblock Multidimensional harmonic retrieval via coupled canonical polyadic
  decomposition—part ii: Algorithm and multirate sampling.
\newblock {\em IEEE Transactions on Signal Processing}, 65(2):528--539, 2016.

\bibitem{so-do}
Mikael S{\o}rensen, Ignat Domanov, and Lieven De~Lathauwer.
\newblock Coupled canonical polyadic decompositions and multiple shift
  invariance in array processing.
\newblock {\em IEEE Transactions on Signal Processing}, 66(14):3665--3680,
  2018.

\bibitem{sovan}
Mikael S{\o}rensen, Frederik Van~Eeghem, and Lieven De~Lathauwer.
\newblock Blind multichannel deconvolution and convolutive extensions of
  canonical polyadic and block term decompositions.
\newblock {\em IEEE Transactions on Signal Processing}, 65(15):4132--4145,
  2017.

\bibitem{vollgraf2006quadratic}
Roland Vollgraf and Klaus Obermayer.
\newblock Quadratic optimization for simultaneous matrix diagonalization.
\newblock {\em IEEE Transactions on Signal Processing}, 54(9):3270--3278, 2006.

\bibitem{objfunc}
Wenwu Wang, Saeid Sanei, and Jonathon Chambers.
\newblock Penalty function-based joint diagonalization approach for convolutive
  blind separation of nonstationary sources, Jan 2005.

\bibitem{Weyl1912}
H.~Weyl.
\newblock Das asymptotische verteilungsgesetz der eigenwerte linearer
  partieller differentialgleichungen (mit einer anwendung auf die theorie der
  hohlraumstrahlung).
\newblock {\em Mathematische Annalen}, 71:441--479, 1912.

\bibitem{inacc}
A.~{Yeredor}.
\newblock Non-orthogonal joint diagonalization in the least-squares sense with
  application in blind source separation.
\newblock {\em IEEE Transactions on Signal Processing}, 50(7):1545--1553, 2002.

\end{thebibliography}


\end{document}